\documentclass[pdflatex,sn-mathphys-num]{sn-jnl}% Math and Physical Sciences Numbered Reference Style
\usepackage{graphicx}%
\usepackage{multirow}%
\usepackage{amsmath,amssymb,amsfonts}%
\usepackage{amsthm}%
\usepackage{mathrsfs}%
\usepackage[title]{appendix}%
\usepackage{xcolor}%
\usepackage{textcomp}%
\usepackage{manyfoot}%
\usepackage{booktabs}%
\usepackage{algorithm}%
\usepackage{algorithmicx}%
\usepackage{algpseudocode}%
\usepackage{listings}%
\theoremstyle{thmstyleone}%
\newtheorem{theorem}{Theorem}%  meant for continuous numbers
\theoremstyle{thmstyletwo}%
\newtheorem{remark}{Remark}%
\newtheorem{lemma}{Lemma}%

\theoremstyle{thmstylethree}%

\DeclareMathOperator{\argmin}{argmin}

\begin{document}

\title[Asymptotic theory for extreme value generalized additive models]{Asymptotic theory for extreme value generalized additive model}


\author*[1]{\fnm{Takuma} \sur{Yoshida}}\email{yoshida@sci.kagoshima-u.ac.jp}

\affil*[1]{\orgdiv{Graduate School of Science and Engineering}, \orgname{Kagoshima University}, \orgaddress{\street{1-21-35, Korimoto}, \city{Kagoshima}, \postcode{890-0065}, \state{Kagoshima}, \country{Japan}}}


\abstract{The classical approach to analyzing extreme value data is the generalized Pareto distribution (GPD). 
When the GPD is used to explain a target variable with the large dimension of covariates, the shape and scale function of covariates included in GPD are sometimes modeled using generalized additive models (GAM).
In contrast to many results of applications, no theoretical results have been reported for the hybrid technique of GAM and GPD, which motivates us to develop its asymptotic theory. 
We provide the rate of convergence of the estimator of shape and scale functions, as well as its local asymptotic normality.}




\maketitle

\section{Introduction}\label{sec1}

Generalized Pareto distribution (GPD) is a typical model to express the tail probability of data. 
The extreme value theory (EVT) explains that the GPD can fit the data exceeding some high threshold value.
The GPD contains two parameters: one characterizes the distribution shape whereas the other can be regarded as characterizing the scale. 
Smith (1987) and Drees et al. (2004) established the asymptotic theory of the maximum likelihood estimators of shape and scale parameters. 

To predict the tail probability of a target variable associated with covariate information, we often consider the GPD with shape and scale functions of the covariates, designated as GPD regression. 
Davison and Smith (1990) have proposed the linear models.
Hall and Tajvidi (2000), Ramesh and Davison (2002), and Beirlant et al. (2004) investigated GPD regression in the context of nonparametric smoothing with kernel methods. 
However, the ordinary nonparametric estimator would be drastically unstable when the number of covariates is large: the so-called curse of dimensionality. 
In regression with numerous covariates, some specific modeling would be better instead of fully nonparametric methods.
One efficient approach is the generalized additive model (GAM). 
Hastie and Tibshirani (1986) proposed the GAM, which has been developed by many authors in several regression models, and which has been summarized by Hastie and Tibshirani (1990) and Wood (2017). 
Chavez-Demoulin and Davison (2005) provided the GPD regression with shape and logarithm of scale functions assumed to be GAM. 
Yee and Stephenson (2007), Chavez-Demoulin (2015), Vatter and Chavez-Demoulin (2015), Mhalla et al. (2019), and Youngman (2019) contributed the additive modeling for extreme value analysis. 
Because the GPD regression has two target functions for shape and scale, its additive modeling is sometimes denoted by vector-generalized additive models, as descirbed by Yee (2015). 
The important work on GAM in GPD regression is Youngman (2022), who published the R-package {\sf evgam}. Consequently, everyone can easily use GAM in GPD for extreme value data analysis.  
The implementation of {\sf evgam} is related closely to the famous R-package {\sf mgcv} (Wood (2011, 2017)). 
Nevertheless, theoretical investigations of GPD regression with GAM are lacking in the relevant literature. 
This lack of investigation to date motivates us to establish its asymptotic theory. 
The key method implemented in {\sf evgam} is the penalized spline method. 
In mean regression, the asymptotic results of spline-based GAM have been developed by Wang and Yang (2007), Yoshida and Naito (2014), Liu et al. (2011), and Liu et al. (2013). 
Therefore, the asymptotic study of GAM in GPD regression is also an important issue in terms of developing GAM versatility. 
As represented herein, we show the asymptotic $L_2$-rate and $L_\infty$-rate of convergences of the GAM estimator of shape and scale functions. 
The local asymptotic normality of the estimators is also obtained. 
 
The rest of this paper is organized as follows. 
Basic conditions of EVT are presented in section 2.
Section 3 provides the estimator of shape and scale function under GAM in GPD regression. 
Main results are presented in section 4. 
We first introduce the mathematical conditions to obtain the asymptotic theory for the estimator in section 4.1. 
Section 4.2 presents the $L_2$ and $L_\infty$-rate of convergence of the estimator as well as local asymptotic normality. 
Section 5 concludes the paper. 
All proofs of theorems and related lemmas are described in the Appendix.
This study specifically examines mathematical results of the estimator under GAM in GPD regression. 
Its numerical performance can be confirmed easily via {\sf evgam}. For that reason, we omit the numerical study.

\section{Preliminaries}

\subsection{Extreme value theory}

We first review the EVT for univariate random variable.
Let $Y\in\mathbb{R}$ be the random variable with distribution function $F$, denoted by $F(y)=P(Y<y)$ for $y\in\mathbb{R}$. 
In the EVT, if there exist sequences $a_n$ and $b_n$ and $\gamma\in\mathbb{R}$ such that $F^n(a_n y+b_n)\rightarrow G(y|\gamma)=\exp[-(1+\gamma y)^{-1/\gamma}]$ for $1+\gamma y>0$, it is said that $F$ belongs to the maximum domain of attraction of distribution $G(y\mid \gamma)$, as denoted by $F\in{\cal D}(G(\cdot|\gamma))$. 
Note that if $\gamma=0$, $G(y\mid 0)=\exp[-\exp[-y]]$.

For some threshold value $w\in\mathbb{R}$, we define 
$$
F_w(y)=\frac{F(w+y)-F(w)}{1-F(w)}
$$
and the GPD function as 
$$
H(y|\gamma)= 1-(1+\gamma y)^{-1/\gamma}
$$
with $1+\gamma y>0$ and parameter $\gamma\in\mathbb{R}$. 
If $\gamma=0$, then we set $H(y|0)=\lim_{\gamma\rightarrow 0}H(y|\gamma)=1-e^{-y}$. 
Then, it is well known that $F\in{\cal D}(G(\cdot|\gamma))$ if and only if there exists a sequence $\sigma_w\in\mathbb{R}_+$ such that for any $y\in(0,y^*-w)$, 
\begin{eqnarray}
\lim_{w\rightarrow y^*} \left|F_w(y)-H\left(\left.\frac{y}{\sigma_w}\right|\gamma\right)\right|\rightarrow 0, \label{Dom1}
\end{eqnarray}
where $y^*=\sup\{t: F(t)<1\}$ (e.g. Theorem 1.2.5 of de Haan and Ferreira (2006)). 
From the above, it is readily apparent that if $\gamma>0$, $y^*=\infty$ and $y^*$ is finite for $\gamma<0$. 
For $\gamma=0$, $y^*$ can be obtained as both finite or infinite. 
For this paper, we only consider the case that $y^*=\infty$ if $\gamma=0$. 
In addition, $\sigma_w$ can be taken as
\begin{eqnarray}
\left\{
\begin{array}{ll}
\underset{w\rightarrow\infty}{\lim} \displaystyle\frac{\sigma_w}{w}=\gamma,& \gamma>0,\\
\underset{w\rightarrow y^*}{\lim} \displaystyle\frac{\sigma_w}{y^*-w}=-\gamma,& \gamma<0,\\
\underset{w\rightarrow\infty}{\lim}  \sigma_w=\sigma,& \gamma=0,\\
\end{array}
\right. \label{ScaleGamma}
\end{eqnarray}
where $\sigma>0$ is some constant (Theorem 1.2.5 in de Haan and Ferreira 2006). 

To predict the probability using the GPD model, the parameters $(\gamma,\sigma_w)$ are needed to be estimated. 
Smith (1987) and Drees et al. (2004) investigated the maximum-likelihood estimator of $(\gamma,\sigma_w)$ and its asymptotic result. 
To establish the asymptotic theory of the estimator of $(\gamma,\sigma_w)$, (\ref{Dom1}) should be modified in the context of the second-order condition of EVT. 

Let $F^{-1}(x)=\inf\{y: F(y)\geq x\}$ and let $U(t)= F^{-1}(1-1/t)$. 
Then, $U(t)\rightarrow y^*$ as $t\rightarrow\infty$. 
According to Theorem 1.1.6 presented by de Haan and Ferreira (2006), $F\in{\cal D}(G(\cdot|\gamma))$ if and only if there exists a function $a:\mathbb{R}_+\rightarrow\mathbb{R}_+$ such that $\{U(tx)-U(t)\}/a(t)\rightarrow (x^\gamma-1)/\gamma$ as $t\rightarrow \infty$ and
$a(t)=\sigma_w$ with $w= U(t)$.
Assuming that function $A:\mathbb{R}_+\rightarrow \mathbb{R}_+$ and $\tilde{Q}:\mathbb{R}_+\rightarrow \mathbb{R}_+$ exist such that
\begin{eqnarray}
\lim_{t\rightarrow\infty}\left|\frac{\frac{U(tx)-U(t)}{a(t)}-\frac{x^\gamma-1}{\gamma}}{A(t)}-\tilde{Q}(x|\gamma,\rho)\right|=0, \label{Dom2}
\end{eqnarray}
where $\rho\leq 0$ is the so-called second-order parameter, 
$$
\tilde{Q}(x|\gamma,\rho)=\frac{1}{\rho}\left(\frac{x^{\gamma+\rho-1}-1}{\gamma+\rho}-\frac{x^\gamma-1}{\gamma}\right),
$$
and $A(t)\rightarrow 0$ as $t\rightarrow\infty$. 
Then, Theorem 2.3.3 of de Haan and Ferreira (2006) demonstrates that $A$ is $\rho$-regularly varying function at $t\rightarrow \infty$, i.e., for all $x\in\mathbb{R}_+$, $A(xt)/A(t)\rightarrow x^{\rho}$ as $t\rightarrow\infty$. 
If $\gamma=0$ or $\rho=0$, then we obtain $\tilde{Q}(x|\gamma,\rho)$ as $\gamma\rightarrow 0$ or $\rho\rightarrow 0$.
By theorem 2.3.8 of de Haan and Ferreira (2006), under (\ref{Dom2}), we obtain 
\begin{eqnarray}
\lim_{w\rightarrow y^*}\left|\frac{F_w(y)-H(y/\sigma_w|\gamma)}{\alpha(w)}-Q(y/\sigma_w|\gamma,\rho) \right|=0 \label{Dom3}
\end{eqnarray}
where $Q(y|\gamma,\rho)=\bar{H}(y|\gamma)^{1+\gamma}\tilde{Q}(\bar{H}^{-1}(y|\gamma)|\gamma,\rho)$, $\bar{H}(y|\gamma)=1-H(y|\gamma)$ and $\alpha(w)=A(1/(1-F(w)))$.
When (\ref{Dom3}) is satisfied, then $F$ is said to belong to the domain of attraction of $G(\cdot|\gamma)$ with second-order parameter $\rho$. 
The asymptotic behavior of the estimator of $(\gamma,\sigma_w)$ is dependent not only $(\gamma,\sigma_w,w)$ but also on $\rho$ and $\alpha(\cdot)$. 
Therefore, the second-order condition of EVT is an important assumption to examine the asymptotic theory for the estimator of $(\gamma,\sigma_w)$.

\begin{remark}
\textup{In (\ref{ScaleGamma}), the original property of scale parameter for $\gamma=0$ is $d \sigma_w/d w\rightarrow 0$ (Theorem 1.2.5 of de Haan and Ferreira 2006). 
That is, for $\gamma=0$, we also allow $\sigma_w=\sigma\log w$ with some constant $\sigma>0$ and $\sigma_w\rightarrow 0$. 
However, if we consider the general condition that $d \sigma_w/dw\rightarrow 0$, the discussion for $\gamma=0$ becomes more complicated (see, Zhou 2009).
Therefore, for the study described herein, we emphasize only the simple case in which $\sigma_w$ converges to constant: $H(y/\sigma_w|0)\approx 1-e^{-y/\sigma}$.}
\end{remark}

\subsection{Extreme value theory in regression}

We next extend the univariate EVT discussed in the preceding section to conditional EVT. 
Let $(Y^*, X, Z)$ be triplet random variables with response $Y^*\in\mathbb{R}$ and covariates $X=(X^{(1)},\ldots, X^{(p)})^\top \in{\cal X}\subset \mathbb{R}^p$ and $Z=(Z^{(1)},\ldots, Z^{(d)})^\top \in{\cal Z}\subset{R}^{d}$. 
Here, ${\cal X}$ and ${\cal Z}$ are assumed to be a compact set. 
Let $F(y|x,z)=P(Y^*<y|X=x,Z=z)$ be a conditional distribution function of $Y^*$ given $(X, Z)=(x, z)=(x^{(1)},\ldots,x^{(p)},z^{(1)},\ldots,z^{(d)})^\top \in{\cal X}\times{\cal Z}$. 
The covariate dependent threshold function is denoted by $\tau(w|x,z)$, where $w$ is some sequence which controls the level of threshold function. 
We choose $\tau$ so that for any $(x,z)\in{\cal X}\times{\cal Z}$, $\tau(w|x,z)\rightarrow y^*(x,z)$ as $w\rightarrow \infty$, where $y^*(x,z)=\sup\{t:F(t|x,z)<1\}$. 
Also, we let $F_{w,\tau}(y|x,z)= \{F(\tau(w|x,z)+y|x,z)-F(\tau(w|x,z)|x,z)\}/\{1-F(\tau(w|x,z)|x,z)\}$. 
For simplicity, we write $F_{w,\tau}(y|x,z)=F_{w}(y|x,z)$. 
Apparently, $F_{w}(y|x,z)= P(Y^*<\tau(w|x,z)+y| Y^*>\tau(w|x,z),X=x, Z=z).$
The typical choices of $\tau$ are constant $\tau(w|x,z)=w$ and conditional quantile function $q(a|x,z)$ with quantile level $a\in(0,1)$.

As an extension of (\ref{Dom1}) to regression version, we assume that, for any $(x,z)\in{\cal X}\times{\cal Z}$, $F(\cdot|x,z)\in{\cal D}(G(\cdot|\gamma_0(x,z)))$ with $\gamma_0:{\cal X}\times{\cal Z}\rightarrow\mathbb{R}$. Moreover, there exists some function $\sigma_{w,\tau}^\dagger:{\cal X}\times{\cal Z}\rightarrow\mathbb{R}_+$ such that 
$$
\lim_{w\rightarrow \infty}\sup_{(x,z)\in{\cal X}\times{\cal Z}}\left|F_w(y|x,z)-H\left(\frac{y}{\sigma_{w,\tau}^\dagger(x,z)}|\gamma_0(x,z)\right)\right|=0.
$$
For simplicity, $\sigma_{w,\tau}^\dagger$ is denoted by $\sigma_w^\dagger$ below. 
Similarly to the preceding section, it is apparent that $y^*(x,z)=\infty$ if $\gamma_0(x,z)>0$, whereas $y^*(x,z)$ is finite when $\gamma_0(x,z)<0$. 
For $\gamma_0(x,z)=0$, we assume that $y^*(x,z)=\infty$. 
Furthermore, we can obtain that, for any $(x,z)\in{\cal X}\times{\cal Z}$,
\begin{eqnarray}
\left\{
\begin{array}{ll}
\displaystyle\lim_{w\rightarrow\infty} \frac{\sigma_{w}^\dagger(x,z)}{\tau(w|x,z)}=\gamma_0(x,z),& \gamma_0(x,z)>0,\\
\displaystyle\lim_{w\rightarrow \infty} \frac{\sigma_{w}^\dagger(x,z)}{y^*(x,z)-\tau(w|x,z)}=-\gamma_0(x,z),& \gamma_0(x,z)<0,\\
\displaystyle \lim_{w\rightarrow\infty} \sigma_{w}^\dagger(x,z)=\sigma^\dagger(x,z),& \gamma_0(x,z)=0\\
\end{array}
\right. \label{SigSeq}
\end{eqnarray}
The above is found for some function $\sigma^\dagger:{\cal X}\times{\cal Z}\rightarrow\mathbb{R}_+$ independent from $w$.
As the second-order condition of conditional EVT, we assume that 
\begin{eqnarray}
&&\lim_{w\rightarrow \infty}\sup_{(x,z)\in{\cal X}\times{\cal Z}}\left|\frac{F_w(y|x,z)-H\left(\frac{y}{\sigma_{w}^\dagger(x,z)}|\gamma_0(x,z)\right)}{\alpha(\tau(w|x,z)|x,z)}-Q\left(\frac{y}{\sigma_{w}^\dagger(x,z)}|\gamma_0(x,z),\rho(x,z)\right) \right|\nonumber\\
&&=0 \label{DomCon1}
\end{eqnarray}
for some function $\rho(x,z)\leq 0$ and $\alpha(\tau(w|x,z)|x,z)$ satisfying $\alpha(\tau(w|x,z)|x,z)\rightarrow 0$ as $w\rightarrow \infty$.

\section{Extreme Value Generalized Additive Models}

This section provides an estimation method of shape and scale functions. 

\subsection{Peak over threshold}
Let $\{(Y^*_i,X_i,Z_i):i=1,\ldots,N\}$ be an $i.i.d.$ random sample from the same distribution as $(Y^*,X,Z)$, where $Y_i^*\in\mathbb{R}$, $X_i=(X_i^{(1)},\ldots,X_i^{(p)})^\top \in{\cal X}$ and $Z_i=(Z_i^{(1)},\ldots,Z_i^{(d)})^\top \in{\cal Z}$. 
For a given threshold function $\tau(w|x,z)$, we let $Y_i=\max\{Y_i^*-\tau(w|X_i,Z_i), 0\}$.
The method of estimating unknown objects using the data exceeding threshold is the so-called peak over threshold.
The conditional distribution of $Y_i$ given $Y_i>0$ and $(X_i,Z_i)=(x,z)$ is $F_w(y|x,z)=P(Y_i<y|Y_i>0,x,z)$. 
Because $F_w(y|x,z)\approx H(y/\sigma_{w}^\dagger(x,z)|\gamma_0(x,z))$, we estimate $(\gamma_0,\sigma^\dagger_{w})$ based on the GPD $H(y/\sigma_{w}^\dagger(x,z)|\gamma_0(x,z))$. 

Here, we provide the likelihood-based estimation method.
The density function from $H$ with arbitrary shape and scale function $(\gamma,\sigma)$ is obtained as
\begin{eqnarray*}
h(y|\gamma(x,z),\sigma(x,z)):= \frac{d}{d y}H\left(\frac{y}{\sigma(x,z)}|\gamma(x,z)\right)&=& \frac{1}{\sigma(x,z)}\left(1+\gamma(x,z)\frac{y}{\sigma(x, z)}\right)^{-1/\gamma(x, z)-1}.
\end{eqnarray*}
When $\gamma(x,z)=0$, $h(y|0,\sigma(x,z))=\lim_{t \rightarrow 0} h(y|t,\sigma(x,z))=\{1/\sigma(x,z)\}\exp[-y/\sigma(x,z)]$. 
The log-likelihood for $(\gamma,\sigma)$ is 
\begin{eqnarray*}
\sum_{i=1}^N \log h(Y_i|\gamma(X_i,Z_i),\sigma(X_i,Z_i))I(Y_i>0).
\end{eqnarray*}
Beirlant and Goegebeur (2004) examined the nonparametric estimation of $(\gamma,\sigma)$ using the kernel-weighted log-likelihood approach. 
However, if the dimension of covariate is large, then fully nonparametric estimation becomes affected adversely by the curse of dimensionality, which leads to poor estimation. 
To avoid this outcome, we introduce the GAM and penalized log-likelihood estimation.

\subsection{Generalized additive model}

For fixed point $x=(x^{(1)},\ldots,x^{(p)})^\top \in{\cal X}$ and $z=(z^{(1)},\ldots,z^{(d)})^\top \in{\cal Z}$, $\gamma$ and $\sigma$ are modeled using the additive model as
$$
\gamma(x,z)=\sum_{j=1}^p \beta_j x^{(j)} + \sum_{j=1}^d g_j(z^{(j)}) =\beta^\top  x+ \sum_{j=1}^d g_j(z^{(j)}) 
$$
and 
$$
\log \sigma(x,z)=\sum_{j=1}^p u_j x^{(j)} + \sum_{j=1}^d s_j(z^{(j)})= u^\top x +\sum_{j=1}^d s_j(z^{(j)}),
$$
where $\beta=(\beta_1,\ldots,\beta_p)^\top \in\mathbb{R}^p$ and $u=(u_1,\ldots,u_p)^\top \in\mathbb{R}^p$ are unknown parameter vectors, and $g_j, s_j :\mathbb{R}\rightarrow\mathbb{R}$ are unknown univariate nonparametric functions. 
For convenience, we assume that $X_i^{(1)} \equiv 1 (i=1,\ldots,n)$. 
Therefore, $\beta_1$ and $u_1$ respectively represent intercept parameters for $\gamma$ and $\log\sigma$.
To obtain identifiability of nonparametric function, we assume that $E[g_j(Z_{i}^{(j)})]=E[s_j(Z_{i}^{(j)})]=0$ for $j=1,\ldots,d$. 
The model above in GPD regression is the so-called extreme value generalized additive models, as presented by Chavez-Demoulin and Davison (2005) and by Youngman (2019). 

We provide the estimation method of $(\beta,u,g_1,\ldots,g_d,s_1,\ldots,s_d)$. 
In the following, for simplicity, the support of $Z$ is set as ${\cal Z}=[0, 1]^d$; that is, $Z_i^{(j)}\in[0,1]$ for all $j=1,\ldots,d$. 
The nonparametric functions $g_j, s_j$'s are approximated by the $B$-spline model. 
Let $0=\kappa_0<\kappa_1<\cdots<\kappa_{K+1}=1$ be the sequence of knots. 
In addition, for some $\xi>0$, we define another $2\xi$ knots as $\kappa_{-\xi}=\cdots=\kappa_{-1}=\kappa_0$ and $\kappa_{K+1}=\kappa_{K+2}\cdots=\kappa_{K+\xi+1}$. 
For simplicity, we assume that the location of knots is equidistant; that is, $\kappa_j-\kappa_{j-1}=1/(K+1)$, but this can be relaxed (see, (3.1) of Xiao 2019). 
Then, we let $\{\psi^{[\xi]}_0(\cdot),\ldots,\psi^{[\xi]}_{K+\xi}(\cdot)\}$ be $\xi$th degree or $(\xi+1)$th order $B$-spline bases, where $\psi^{[\xi]}_k:[0,1]\rightarrow \mathbb{R}_+$. 
The definition and some basic properties of $B$-spline bases are clarified in work reported by de Boor (2001). 
We next transform the ordinary $B$-spline bases $\{\psi^{[\xi]}_0(\cdot),\ldots,\psi^{[\xi]}_{K+\xi}(\cdot)\}$ to the normalized $B$-spline bases (see, Liu et al. 2011). 
For $j=1,\ldots,d$ and $k=1,\ldots,K+\xi-1$, we define 
$$
B^{[\xi]}_{j,k}(\tilde{z})= \frac{\bar{\psi}^{[\xi]}_{j,k}(\tilde{z})}{\|\bar{\psi}^{[\xi]}_{j,k}\|} ,\ \ k= 1,\ldots,K+\xi, 
$$
where $\bar{\psi}^{[\xi]}_{j,k}(\tilde{z})=\psi^{[\xi]}_k(\tilde{z})- (\phi_{j,k}/\phi_{j,k-1})\psi^{[\xi]}_{k-1}(\tilde{z})$, $\phi_{j,k}=E[\psi_k^{[\xi]}(Z_{i}^{(j)})]$ and $\|\cdot\|$ denotes the Euclidean norm. 
For $\tilde{z}\in[0,1]$, the normalized $B$-spline model is defined as 
$$
\bar{g}_j(\tilde{z})= \sum_{k=1}^{K+\xi} B^{[\xi]}_{j,k}(\tilde{z}) b_{j,k}=B_j(\tilde{z})^\top b_j
$$
and 
$$
\bar{s}_j(\tilde{z})= \sum_{k=1}^{K+\xi} B^{[\xi]}_{j,k}(\tilde{z}) c_{j,k}=B_j(\tilde{z})^\top  c_j,
$$
where 
$B_j(\tilde{z})=(B^{[\xi]}_{j,1}(\tilde{z}),\ldots, B^{[\xi]}_{j,K+\xi}(\tilde{z}))^\top $, and $b_j=(b_{j,1},\ldots,b_{j,K+\xi})^\top $ and $c_j=(c_{j,1},\ldots,c_{j,K+\xi})^\top $ are $(K+\xi)$-unknown parameter vectors. 
By the definition of $B^{[\xi]}_{j,k}$, we can confirm easily that $E[\bar{g}_j(Z^{(j)})]=E[\bar{s}_j(Z^{(j)})]=0$. 
We then consider that the additive functions $g_j,s_j$ are approximated by the normalized $B$-spline model: $g_j\approx \bar{g}_j$ and $s_j\approx \bar{s}_j$ for $j=1,\ldots,d$. 
Then, for fixed point $x\in\mathbb{R}^p$ and $z=(z^{(1)},\ldots,z^{(d)})^\top \in[0,1]^d$, $\gamma$ and $\sigma$ are approximated by 
$$
\bar{\gamma}(x,z)=x^\top \beta +\sum_{j=1}^d \bar{g}_j(z^{(j)})=x^\top \beta +\sum_{j=1}^d B_j(z^{(j)})^\top b_j
$$
and 
$$
\log \bar{\sigma}(x, z)=x^\top u +\sum_{j=1}^d \bar{s}_j(z^{(j)})= x^\top u +\sum_{j=1}^d B_j(z^{(j)})^\top c_j. 
$$ 
Accordingly, our purpose is to estimate the parameter vector $(\beta, u, b, c)$, where $b=(b_1^\top ,\ldots,b_d^\top )^\top \in\mathbb{R}^{d(K+\xi)}$ and $c=(c_1^\top ,\ldots,c_d^\top )^\top \in\mathbb{R}^{d(K+\xi)}$.

The minus log-likelihood of $(\beta, u, b, c)$ is 
$$
\ell(\beta, u, b, c)= -\frac{1}{N}\sum_{i=1}^N \log h(Y_i|\bar{\gamma}(X_i,Z_i), \bar{\sigma}(X_i,Z_i))I(Y_i>0).
$$
Then, $(\beta, u, b, c)$ is estimated by minimizing the penalized (minus) log-likelihood as
\begin{eqnarray}
\ell_{pen}(\beta, u, b, c)=\ell(\beta, u, b, c) +\sum_{j=1}^d\left\{ \lambda_j \int_0^1 \{\bar{g}^{(m)}_j(z)\}^2dz
+\nu_j \int_0^1\{\bar{s}^{(m)}_j(z)\}^2dz\right\}, \label{Lpen}
\end{eqnarray}
where $\lambda_j$'s and $\nu_j$'s are smoothing parameters and where $m$ is some integer smaller than $\xi$.
In fact, if $m>\xi$, then $\bar{\gamma}^{(m)}_j=\bar{s}^{(m)}_j\equiv 0$. 
In practice, $(\xi,m)=(3,2)$ is standard.
The estimator obtained from (\ref{Lpen}) is denoted by $(\hat{\beta}, \hat{u}, \hat{b}, \hat{c})$, where $\hat{b}=(\hat{b}_1^\top ,\ldots,\hat{b}_d^\top )^\top $ and $\hat{c}=(\hat{c}_1^\top ,\ldots,\hat{c}_d^\top )^\top $. 
From these estimators, we construct
$$
\hat{\gamma}(x,z)= x^\top \hat{\beta}+ \sum_{j=1}^d \hat{g}_j(z^{(j)})
$$
and 
$$
\hat{\sigma}(x, z)=\exp\left[x^\top \hat{u}+ \sum_{j=1}^d \hat{s}_j(z^{(j)}) \right],
$$
where $\hat{g}_j(z^{(j)})=B^{[\xi]}(z^{(j)})^\top  \hat{b}_j$ and $\hat{s}_j(z^{(j)})=B^{[\xi]}(z^{(j)})^\top  \hat{c}_j$.
The estimator above can be implemented via the R-package {\sf evgam}. 
Therefore, numerical performance of $\hat{\gamma}$ and $\hat{\sigma}$ is already guaranteed by {\sf evgam}. 
No result of the theoretical evidence of $\hat{\gamma}$ and $\hat{\sigma}$ has been stated thus far.
We establish the asymptotic theory for $\hat{\gamma}$ and $\hat{\sigma}$ as presented in the next section. 

\begin{remark}
\textup{
The penalty in (\ref{Lpen}) is known as O'Sullivan's penalty (see O'Sullivan 1986). Another commonly used penalty in GAM is the P-spline penalty (Eilers and Marx 1996, Marx and Eilers 1998).
Roughly speaking, the P-spline penalty is given as $D_{m,K}^\top R_m D_{m,K}$, replacing $R_m$ with identity matrix and $D_{m,K}$ with $D_{m,K-1}$ in Lemma 4 of Appendix B.  
Therefore, we specifically examine O'Sullivan's penalty herein. 
Both penalties are implemented in the R packages {\sf mgcv} and {\sf evgam}.
The practical performances of above two penalties have been discussed by Wand and Ormerod (2008), Ruppert et al. (2009), and Eilers (2015). 
From a theoretical perspective, Xiao (2019) showed that the asymptotic behavior of the estimator with P-spline is almost similar to that with O'Sullivan's penalty when the smoothing parameters are adjusted.
Similarly, the results presented in the next section can also be applied to estimators with the P-spline penalty with some adjustments to the smoothing parameters settings.
}
\end{remark}

\section{Asymptotic Theory}

\subsection{General condition}

We first define the functional space of the additive model. 
Let 
$$
{\cal A}_j=\{\alpha:[0,1]\rightarrow\mathbb{R} | E[\alpha(Z_{i}^{(j)})]=0, V[\alpha(Z_{i}^{(j)})]<\infty \}.
$$
We further let 
$$
{\cal T}_{A}=\left\{f:{\cal X}\times{\cal Z}\rightarrow\mathbb{R} | f(x,z)= x^\top \beta+ \sum_{j=1}^d q_j(z^{(j)}), q_j\in {\cal A}_j, \beta\in\mathbb{R}^p \right\}.
$$
We then assume that for any $(x,z)\in{\cal X}\times{\cal Z}$, $F(\cdot|x,z)\in {\cal D}(G(\cdot|\gamma_0(x,z)))$ with $\gamma_0\in{\cal T}_{A}$. 
However, although $\log \hat{\sigma}(x,z)$ is constructed via the additive model, we cannot know whether $\log \sigma_{w}^\dagger\in{\cal T}_{A}$ because $\sigma^\dagger_w$ depends on $\tau(w|x,z)$.
Therefore, we use the alternative target scale function to investigate the asymptotic behavior of $\hat{\sigma}$. 
Define the target functions as 
\begin{eqnarray}
(\gamma_0,\sigma_{w0})=\underset{(\gamma,\log\sigma)\in{\cal T}_{A}\times{\cal T}_{A}}{\argmin} -E[\log h(Y|\gamma(X,Z),\sigma(X,Z))]. \label{TrueRisk}
\end{eqnarray}
Then, $\sigma_{w0}$ and $\sigma^\dagger_{w}$ are not equal. However, this model bias cannot be evaluated or improved as long as we consider the GAM. 
Herein, for simplicity, we assume that such bias can be ignored.
Accordingly, from (\ref{SigSeq}), $\sigma_{w0}$ satisfies
\begin{eqnarray}
\left\{
\begin{array}{ll}
\displaystyle\lim_{w\rightarrow\infty} \frac{\sigma_{w0}(x,z)}{\tau(w|x,z)}=O(1),& \gamma_0(x,z)>0,\\
\displaystyle\lim_{w\rightarrow \infty} \frac{\sigma_{w0}(x,z)}{y^*(x,z)-\tau(w|x,z)}=O(1)& \gamma_0(x,z)<0,\\
\displaystyle \lim_{w\rightarrow\infty} \sigma_{w0}(x,z)=\sigma_0(x,z),& \gamma_0(x,z)=0,\\
\end{array}
\right. \label{SigSeq2}
\end{eqnarray}
where $\log \sigma_0\in{\cal T}_{A}$.

We continue with discussion of the behavior of $\sigma_{w0}$. 
Because $\sigma_{w0}$ depends on $w$, the true coefficients of the parametric part and each additive function of the nonparametric part in $\log \sigma_{w0}$ might also vary with $w$. 
Therefore, for $\gamma_0$ and $\sigma_{w0}$, there exist $\beta_0=(\beta_{01},\ldots,\beta_{0p})^\top , u_{w0}=(u_{w,01},\ldots,u_{w,0p})^\top \in\mathbb{R}^p$ and $g_{0j},s_{w,0j}\in{\cal A}_j (j=1,\ldots,d)$ such that for any $x=(x^{(1)},\ldots,x^{(p)})\in{\cal X}$ with $x^{(1)}=1$ and $z=(z^{(1)},\ldots,z^{(d)})\in{\cal Z}$, 
\begin{eqnarray}
\gamma_0(x,z)=x^\top \beta_0+\sum_{j=1}^d g_{0j}(z^{(j)}) \label{additiveEVI}
\end{eqnarray}
and 
\begin{eqnarray}
\log \sigma_{w0}(x,z)=x^\top u_{w0}+\sum_{j=1}^d s_{w,0j}(z^{(j)}). \label{additiveScale}
\end{eqnarray}
For (\ref{additiveScale}), we assume that coefficients except for intercept and each additive function are independent from $w$, as denoted by $u_{w,0j}=u_{0j}$ for $j=2,\ldots,p$ and $s_{w,0k}=s_{0k}$ for $k=1,\ldots,d$. 
That is, (\ref{additiveScale}) can be written as 
\begin{eqnarray}
\log \sigma_{w0}(x,z)=u_{w,01}+\sum_{j=2}^p u_{0j}x^{(j)}+\sum_{j=1}^d s_{0j}(z^{(j)}). \label{additiveScale2}
\end{eqnarray}
For $\gamma_0(x,z)>0$, if we choose $\tau$ as $\tau(w|x,z)=a_1(w)\tau_1(x,z)$ for some functions $a_1:\mathbb{R}\rightarrow\mathbb{R}$ and $\tau_1:{\cal X}\times{\cal Z}\rightarrow\mathbb{R}$ with $a_1(w)\rightarrow \infty (w\rightarrow\infty)$, (\ref{additiveScale2}) holds with $u_{w,01}=\log a_1(w)+u_0, u_0\in\mathbb{R}$. 
A simple but important example is $a_1(w)=w$ and $\tau_1(x,z) \equiv 1$. 
For $\gamma_0(x,z)<0$, we obtain (\ref{additiveScale2}) when $\tau$ can be written as $y^*(x,z)-\tau(w|x,z)=a_2(w)\tau_2(x,z)$ for some functions $a_2:\mathbb{R}\rightarrow\mathbb{R}$ and $\tau_2:{\cal X}\times{\cal Z}\rightarrow\mathbb{R}$ with $a_2(w)\rightarrow 0 (w\rightarrow\infty)$. 
Because $y^*(x,z)$ is unknown, identifying such $a_2$ and $\tau_2$ is difficult in practice.

\begin{remark}
\textup{
We provide an additional explanation of the assumption of (\ref{additiveScale2}). 
For simplicity, we assume that $\gamma_0(z)>0$ for all $z=(z^{(1)},z^{(2)})\in[0,1]^2$ and (\ref{additiveScale}) as $\sigma_{w0}(z)=\exp[u_{w0}+s_{w,01}(z^{(1)})+s_{w,02}(z^{(2)})]$. 
Then, from (\ref{SigSeq2}), roughly speaking, $s_{w,01}$ might tend to infinity because $\tau(w|z)\rightarrow\infty$ and $\sigma_{w0}\rightarrow \infty$ as $w\rightarrow\infty$. 
However, from the condition $E[s_{w,01}(Z^{(1)})]=0$ for the identifiability of function, if $s_{w,01}(z^*)\rightarrow \infty$ for a point $z^*\in[0,1]$, another point $z^{**}\in[0,1]$ exists such that $s_{w,01}(z^{**})\rightarrow -\infty$. 
Nevertheless, for all $z\in[0,1]^2$, $\sigma_{w0}(z)$ must diverge to infinity as $w\rightarrow\infty$. 
This yields that $s_{w,01}(z^{**})/u_{w0}\rightarrow 0$ as $w\rightarrow\infty$. 
Therefore, $u_{w0}$ should dominate other components as the sequence of $w$. 
Furthermore, $s_{w,0j}=O(1)$ or $s_{w,0j}=o(1)$ can be allowed as long as $\sigma_{w0}\rightarrow\infty$. 
In fact, if $\tau(w|z)=w$, then we can obtain that $u_{w0}=\log w+u_0$ with $u_0\in\mathbb{R}$ and $s_{w,0j}=O(1)$ from (\ref{SigSeq2}).
Consequently, the construction of $s_{w,0j}$ as the sequence of $w$ becomes quite complicated in (\ref{additiveScale}); (\ref{additiveScale2}) is the simplest model with no contradiction as (\ref{SigSeq2}).
}
\end{remark}

As we have said, the behavior of $\sigma_{w0}$ depends on the signature of $\gamma_0$. 
From (\ref{SigSeq2}), if $\gamma_0(x,z)>0$, $\sigma_{w0}(x,z)\rightarrow \infty$, and if $\gamma_0(x,z)<0$, then we have $\sigma_{w0}(x,z)\rightarrow 0$. 
Therefore, if $\gamma_0(x_0,z_0)=0$ at the point $(x_0,z_0)\in{\cal X}\times{\cal Z}$ and $\gamma_0(x,z)$ can take both positive and negative at the neighborhood of $(x_0,z_0)$, then $\sigma_{w0}(x,z)$ is not continuous at $(x_0,z_0)$ although $\log \sigma$ is modeled by continuous additive functions. 
For that reason, when we consider the additive model for the scale function, the sign of the shape function is assumed to be fixed, even if it is unknown.
From this, we consider the following three cases separately: (i) $\inf_{(x,z)\in{\cal X}\times{\cal Z}}\gamma_0(x,z)>0$, (ii) $\sup_{(x,z)\in{\cal X}\times{\cal Z}}\gamma_0(x,z)<0$ and (iii) $\gamma_0(x,z)\equiv 0$ for all $(x,z)\in{\cal X}\times{\cal Z}$. 
One might consider case (iii) to be unrealistic, but if $F(\cdot|x,z)$ is the Gaussian distribution, then it belongs to ${\cal D}(G(\cdot|\gamma_0(x,z)))$ with $\gamma_0(x,z)= 0$. 
If the sign of $\gamma_0$ varies among $(x,z)\in{\cal X}\times{\cal Z}$, then the additive modeling for $\sigma_{w0}$ would not be performed better. 
Such a case is beyond the scope of this study. 
Therefore, we first state following conditions concerned with positive, negative, and zero shape function. 

\begin{enumerate}
\item[(S1)] Constant $\gamma_{min}>0$ exists such that $\gamma_{min}<\gamma_0(x,z)$ for all $(x,z)\in{\cal X}\times{\cal Z}$. 
\item[(S2)] For some $\delta>0$, $-1/(2+\delta)<\gamma_0(x,z)<0$ for all $(x,z)\in{\cal X}\times{\cal Z}$. 
\item[(S3)] For all $(x,z)\in{\cal X}\times{\cal Z}$, $\gamma_0(x,z)=0$. 
\end{enumerate}
We establish the asymptotic theory of $(\hat{\gamma},\hat{\sigma})$ under each condition (S1), (S2), or (S3).

As the true setting, (\ref{DomCon1}), (\ref{TrueRisk}), (\ref{SigSeq2}), and (\ref{additiveScale2}) are assumed. 
In addition, we state the technical conditions to investigate the asymptotic property of the estimator for all cases (S1)--(S3). 
For $a>0$, let ${\cal C}^a$ be the class of functions with $a$th continuously differentiable on [0,1]. 
Define 
\[
\lambda =\max_j \lambda_j,\ \ \nu=\max_j \nu_j.
\]
For simplicity, we write 
\[
p_N\equiv P(Y>0)=P(Y^*>\tau(w_N\mid X, Z)).
\]

The mathematical conditions to establish the asymptotic result of the estimator are listed below.

\begin{itemize}
\item[(C1)] The log-likelihood function $E[\log h(Y|\gamma(X,Z),\sigma(X,Z))]$ is a concave function with respect to $(\gamma,\sigma)\in{\cal T}_A\times{\cal T}_A$. 
\item[(C2)] For $k=1,\ldots,d$, $g_{0k}\in{\cal A}_k\cap {\cal C}^\zeta$ and $s_{0k}\in {\cal A}_k\cap {\cal C}^\zeta$ with $\zeta=\xi+1$.
\item[(C3)] As $N\rightarrow\infty$, $p_N\rightarrow 0$ and $N p_N \rightarrow\infty$.  
\item[(C4)] In (\ref{DomCon1}), there exists a continuous and bounded function $\alpha_1(x,z)$ such that $\alpha(\tau(w_N\mid x,z)\mid x,z)= \alpha_1(x,z)(1-F(\tau(w_N\mid x,z)\mid x,z))^{-\rho(x,z)}(1+o(1))$ as $N\rightarrow\infty$. 
In addition, there exists a constant $\rho<0$ such that $\sup_{(x,z)\in{\cal X}\times{\cal Z}}\rho(x,z)\leq \rho<0$. 
\item[(C5)] For the number of knots $K$ and smoothing parameters $\lambda$ and $\nu$, satisfy $K\rightarrow\infty$, $K\log(N)/Np_N\rightarrow 0$, $\lambda p_N^{-1} K^{2m}=O(1)$ and $\nu p_N^{-1}K^{2m}=O(1)$ as $N\rightarrow\infty$. 
\end{itemize}

(C1) is the standard condition in the log-likelihood type estimation method. 
If $\gamma$ and $\sigma$ are constant, then (C1) can be reduced to $\gamma>-1/2$ (Smith 1987, Drees et al. 2004). 
On the other hand, (C1) is the minimal condition in GAM in GPD regression because the covariate information is included in the Fisher information matrix of $E[\log h(Y|\gamma(X,Z),\sigma(X,Z))]$. 
(C2) defines the smoothness of true additive function. 
This is also standard when the order of $B$-spline basis is $\zeta=\xi+1$ (see, Xiao 2019). 
(C3) explains the condition of sample size of data exceeding threshold. 
For the effective sample size $n=\sum_{i=1}^N I(Y_i>0)$, it can be shown easily that $E[n]/N=P(Y>0)=p_N$ and $V[n/N]=p_N(1-p_N)/N\rightarrow 0$ under (C3). 
From this result, we obtain $n\approx N p_N$ and $n\rightarrow\infty$ but $n/N\rightarrow 0$ as $N\rightarrow \infty$. 
Such $n$ is the so-called intermediate sequence in EVT as presented in Section 2 of  de Haan and Ferreira (2006). 
In this sense, (C3) is natural in the EVT.
(C4) controls the second-order bias in the EVT. 
We provide some justification of (C4) in Remark \ref{rem2}. 
The case $\rho(x,z)=0$ is removed in (C4). 
If $\rho(x,z)=0$, then this corresponds to $\alpha(\tau(w_N\mid x,z)\mid x,z)=O(\log P(Y_i>0))$ for example. 
This rate is too slow to explain the efficiency of the estimator. 
(C5) examines the rates of tuning parameters to obtain the optimal rate of convergence of the estimator. 
Xiao (2019) uses similar conditions of the number of knots and smoothing parameters. 
Remark 6 in the following subsection is also relevant.

\begin{remark}\label{rem2}
\textup{
We note in (C4). 
For $t>0$, let $U(t\mid x,z)=F^{-1}(1-1/t\mid x,z)$. 
Then, from Theorems 2.3.3 and 2.3.8 presented by de Haan and Ferreira (2006), the equivalence condition to (\ref{DomCon1}) is that there exists a function $\Lambda$ such that 
$$
\lim_{t\rightarrow \infty}\sup_{(x,z)\in{\cal X}\times{\cal Z}}\left| \frac{\frac{U(t\eta\mid x,z)-U(t\mid x,z)}{a_t(x,z)}-\frac{(\eta^{-\gamma_0(x,z)}-1)}{\gamma_0(x,z)}}{\Lambda(t\mid x,z)}-\tilde{Q}(\eta\mid \gamma_0(x,z),\rho(x,z))\right|  \rightarrow 0
$$ 
for all $\eta>0$, where $a_t(x,z)=\sigma_w^\dagger(x,z)$ with $\tau(w\mid x,z)=U(t\mid x,z)$ and the function $\Lambda(t\mid x,z)$ satisfying $\Lambda(\eta t \mid x,z)/\Lambda(t \mid x,z)\rightarrow \eta^{\rho(x,z)}$ as $t\rightarrow\infty$.  
As one of typical choice of $\Lambda$, we use $\Lambda(\eta t \mid x,z)=\alpha_1(x,z) t^\rho(x,z)$ with some bounded function $\alpha_1(x,z)$, which is given in (C4).
Although we can consider another choice of $\Lambda$, e.g., $\Lambda(t \mid x, z) = \alpha_1(x, z)t^{\rho(x,z)} \log t$, these cases are omitted from this paper.
From Theorem 2.3.8 presented by de Haan and Ferreira (2006), we obtain $\alpha(t\mid x,z)=\Lambda(1/\{1-F(t\mid x,z)\}\mid x,z)$, which implies that $\alpha(\tau(w\mid x,z)\mid x,z)=\alpha_1(x,z) \{1-F(\tau(w_N\mid x,z)\mid x,z)\}^{-\rho(x,z)}$. 
}
\end{remark}

\subsection{Main results}

For any function $r:{\cal X}\times{\cal Z}\rightarrow\mathbb{R}$, we define $L_2$-norm of $r$ as $\|r\|_{L_2}= \sqrt{E[r(\tilde{X},\tilde{Z})^2]}$, where $(\tilde{X},\tilde{Z})$ is the random variable having the same distribution as $(X,Z)$ independently.
If $r$ depends on sample $\{(Y_i,X_i,Z_i):i=1,\ldots,n\}$, then $E$ takes the expectation for not only $(\tilde{X},\tilde{Z})$ but also the sample. 
The $L_\infty$-norm of $r$ is defined as $\|r\|_{\infty}=\sup_{(x,z)\in{\cal X}\times{\cal Z}}|r(x,z)|$. 
We first describe the $L_2$-rate of convergence of the estimator.

\begin{theorem}\label{ThmL2}
Suppose that {\rm (C1)--(C5)}. 
In each scenario {\rm (S1), (S2)} or {\rm (S3)}, as $N\rightarrow \infty$,
\begin{eqnarray*}
\|\hat{\gamma}-\gamma_0\|_{L_2}&\leq& O\left(\sqrt{\frac{K}{Np_N}}\right)+O(K^{-m})+O\left(p_N^{-\rho}\right),\\
\left\|\log \hat{\sigma}-\log \sigma_{w0} \right\|_{L_2}
&\leq& O\left(\sqrt{\frac{K}{N p_N}}\right)+O(K^{-m})+O\left(p_N^{-\rho}\right).
\end{eqnarray*}
Under the optimal rate of number of knots $K= O(\{N p_N\}^{1/(2m+1)})$, 
\begin{eqnarray*}
\|\hat{\gamma}-\gamma_0\|_{L_2}&\leq& O\left((Np_N)^{-m/(2m+1)}\right)+O\left(p_N^{-\rho}\right),\\
\left\|\log \hat{\sigma}-\log \sigma_{w0} \right\|_{L_2}
&\leq&O\left((Np_N)^{-m/(2m+1)}\right)+O\left(p_N^{-\rho}\right).
\end{eqnarray*}
If we take $p_N=O(N^{-1/(1-2\rho+1/m)})$, then the optimal convergence rates of the estimators are
\begin{eqnarray*}
\|\hat{\gamma}-\gamma_0\|_{L_2}&\leq& O\left(N^{\frac{\rho}{1-2\rho+1/m}}\right),\\
\left\|\log \hat{\sigma}-\log \sigma_{w0} \right\|_{L_2}
&\leq&O\left(N^{\frac{\rho}{1-2\rho+1/m}}\right).
\end{eqnarray*}
\end{theorem}

In the first assertion of Theorem \ref{ThmL2}, the term $O(\{K/(Np_N)\}^{1/2})$ expresses the order of standard deviation of the estimator, whereas $O(K^{-m})$ is the bias occurring from penalized spline method and $O(p_N^{-\rho})$ is the bias came from the approximation of GPD (\ref{DomCon1}). 
The term $O(p_N^{-\rho})$ depends on the second-order parameter $\rho$ given in (C4).
If (C4) is not assumed, then this term $O(p_N^{-\rho})$ must be changed by $E[\alpha(\tau(w_N\mid X,Z)\mid X,Z)]$, which is difficult to be understood. 
Therefore, to attain easy interpretation of the bias, using the effective sample rate and second-order parameter (C4) is helpful.

Drees (2001) shows that the minimax optimal rate of the estimator of positive EVI in one-dimensional data is $O(N^{\rho/(1-2\rho)})$. 
Consequently, the optimal rate in the last assertion of Theorem \ref{ThmL2} is slightly slower than $O(N^{\rho/(1-2\rho)})$ because $m>0$. 
This result is not surprising because the nonparametric estimator has a slower rate than the parametric estimator, as described by Tsybakov (2009) and others.

Next, we investigate the $L_\infty$-convergence of the estimator. 

\begin{theorem}\label{Linfty}
Suppose that {\rm (C1)--(C5)}. 
Under {\rm (S2)}, suppose that \[
\{K\log N\}^{1+\delta/2}/(Np_N)^{\delta/2}\rightarrow 0
\] as $N\rightarrow\infty$. 
Then, in each scenario {\rm (S1), (S2)} or {\rm (S3)}, as $N\rightarrow \infty$,
\begin{eqnarray*}
\|\hat{\gamma}-\gamma_0\|_{L_\infty}&\leq& O\left(\sqrt{\frac{K\log N}{N p_N}}\right)+O(K^{-m})+O\left(p_N^{-\rho}\right),\\
\left\|\log \hat{\sigma}-\log \sigma_{w0} \right\|_{L_\infty}
&\leq& O\left(\sqrt{\frac{K\log N}{Np_N}}\right)+O(K^{-m})+O\left(p_N^{-\rho}\right).
\end{eqnarray*}
Under the optimal rate of number of knots $K= O_P((Np_n/\log N)^{1/(2m+1)})$, 
\begin{eqnarray*}
\|\hat{\gamma}-\gamma_0\|_{L_\infty}&\leq& O\left((Np_N/\log N)^{-m/(2m+1)}\right)+O\left(p_N^{-\rho}\right),\\
\left\|\log \hat{\sigma}-\log \sigma_{w0} \right\|_{L_\infty}
&\leq&O\left((Np_N/\log N)^{-m/(2m+1)}\right)+O\left(p_N^{-\rho}\right).
\end{eqnarray*}
If we take $p_N=O((N/\log N)^{-1/(1-2\rho+1/m)})$, the optimal convergence rates of the estimators are
\begin{eqnarray*}
\|\hat{\gamma}-\gamma_0\|_{L_\infty}&\leq&O\left(\left(\frac{N}{\log N}\right)^{\frac{\rho}{1-2\rho+1/m}}\right),\\
\left\|\log \hat{\sigma}-\log \sigma_{w0} \right\|_{L_\infty}
&\leq&O\left(\left(\frac{N}{\log N}\right)^{\frac{\rho}{1-2\rho+1/m}}\right).
\end{eqnarray*}
\end{theorem}

\begin{remark}
\textup{
For Theorem \ref{Linfty}, under (S2), we need the extra condition of $\{K\log N\}^{1+\delta/2}/(Np_N)^{\delta/2}\rightarrow 0$.
This condition is necessary for the technical reason to prove $L_\infty$-convergence of the estimator (see Appendix C). 
When $K=O((Np_N/\log N)^{1/(2m+1)})$ is used, we have 
$$
\{K\log N\}^{1+\delta/2}=O\left((Np_N)^{\frac{2+\delta}{4m+2}}(\log N)^{\frac{2m(2+\delta)}{4m+2}}\right).
$$
Then, $\{K\log N\}^{1+\delta/2}/(Np_N)^{\delta/2}\rightarrow 0$ holds if $\delta>1/m$, which implies that (S2) can be replaced with $\gamma_0(x,z)>-m/(2m+1)$ for some $m\in\mathbb{N}$. 
For the spline method, we often use $m=2$, which indicates that (S2) becomes $0>\gamma_0(x,z)>-2/5$. 
Therefore, when $\gamma_0$ is estimated using the nonparametric method, the bound condition of the negative shape function is somewhat stronger than that for the non-regression case: $\gamma_0>-1/2$ (see, Drees et al. 2004). 
}
\end{remark}

Next, the local asymptotic normality of the estimator is analyzed. 
Let $B(Z)=(B_1(Z^{(1)})^\top ,\ldots,B_d(Z^{(d)})^\top )^\top$.
For $(x,z)\in{\cal X}\times{\cal Z}$, let $(p+K+\xi)$-vector $A(x,z)=(x^\top , B(z)^\top )^\top $ and let $2(p+K+\xi) \times 2$ matrix as
$$
D(x,z)=
\left[
\begin{array}{cc}
(x^\top , B(z)^\top )^\top & 0_{p+K+\xi}^\top \\
0_{p+K+\xi}^\top  &(x^\top , B(z)^\top )^\top 
\end{array}
\right]^\top ,
$$
where $0_{p+K+\xi}$ is the $(p+K+\xi)$-zero vector. 
Furthermore, we let
\begin{eqnarray*}
&&\Sigma\\
&&=
\frac{1}{p_N}E\left[
\frac{P(Y>0\mid X,Z)}{2\gamma_0(X,Z)+1}
\left(
\begin{array}{cc}
\frac{2}{\gamma_0(X,Z)+1}A(X,Z)A(X,Z)^\top &\frac{1}{\gamma_0(X,Z)+1}A(X,Z)A(X,Z)^\top \\
\frac{1}{\gamma_0(X,Z)+1}A(X,Z)A(X,Z)^\top & A(X,Z)A(X,Z)^\top 
\end{array}
\right)
\right]\\
&&+\frac{1}{p_N}\Omega_{\gamma,\sigma},
\end{eqnarray*}
where $\Omega_{\gamma,\sigma}$ is defined in the proof of Lemma \ref{GradientEx}.
Because each element of $B(z)$ has an order $O(\sqrt{K})$ and the element of $\Sigma^{-1}$ has $O(1)$, we obtain $D(x,z)^\top \Sigma^{-1}D(x,z)/K=O(1)$. 
We then obtain the result presented below.

\begin{theorem}\label{Norm}
Suppose that {\rm (C1)--(C5)}. 
Furthermore, assume that $(K,w_N)$ satisfies 
$\sqrt{Np_N/K}\{p_N^{-\rho}+K^{-m}\}=\infty$ as $N\rightarrow \infty$.
Then, in each scenario {\rm (S1), (S2)} or {\rm (S3)},
  for any $(x,z)\in{\cal X}\times{\cal Z}$, 
\begin{eqnarray*}
\sqrt{\frac{Np_N}{K}}
\left[
\begin{array}{c}
\displaystyle \hat{\gamma}(x,z)-\gamma_0(x,z)\\
\displaystyle \frac{\hat{\sigma}(x,z)}{\sigma_{w0}(x,z)}-1
\end{array}
\right]
\stackrel{D}{\longrightarrow} N_2\left(0, \lim_{N\rightarrow\infty}D(x,z)^\top \Sigma^{-1}D(x,z)/K\right),
\end{eqnarray*}
as $N\rightarrow \infty$, where $N_2$ means the bivariate normal distribution. 
\end{theorem}

In Theorem \ref{Norm}, the bias disappears because of the choice of $(K,p_N)$.
In fact, the order of the bias of the estimator is $O(K^{-m})+O(p_N^{-\rho})$, whereas the standard deviation has $O(\sqrt{K/Np_N})$. 
From Theorem \ref{ThmL2}, $\sqrt{Np_N/K}\{K^{-m}+p_N^{-\rho}\}=O(1)$ balances the bias and standard deviation of the estimator. 
Consequently, the condition of $(K,p_N)$ in Theorem \ref{Norm} implies that the asymptotic order of the standard deviation of the estimator is slightly larger than that of the bias. 
Because the form of bias is complicated, deriving its consistent estimator is quite challenging. 
Therefore, Theorem \ref{Norm} establishes the variance-dominated asymptotic normality of the estimator $(\hat{\gamma}, \hat{\sigma})$, which facilitates the construction of confidence intervals for the shape and scale functions (see Chapter 4 of Coles, 2001).

Assume a $2p$-square matrix as
$$
\Sigma_{\beta,u}
=
\frac{1}{p_N}E\left[
\frac{P(Y>0\mid X,Z)}{2\gamma_0(X,Z)+1}
\left(
\begin{array}{cc}
\frac{2}{\gamma_0(X,Z)+1}XX^\top & \frac{1}{\gamma_0(X,Z)+1}XX^\top \\
\frac{1}{\gamma_0(X,Z)+1}XX^\top & XX^\top 
\end{array}
\right)
\right].
$$
For any vector $v$, let $\|v\|_{L_2}=\sqrt{E[v^\top v]}$. 
If one wants to examine the linear part of the estimator specifically, then the following result would be helpful. 

\begin{theorem}\label{ParametricPart}
Suppose that {\rm (C1)--(C5)}. 
Furthermore, we choose $(K,p_N)$ such that $(Np_N)^{-1/(2m)}K\rightarrow\infty$ as $N\rightarrow\infty$.
Then, in each scenario {\rm (S1), (S2)} or {\rm (S3)},
\begin{eqnarray*}
\|\hat{\beta}-\beta_0\|_{L_2}&\leq& O((Np_N)^{-1/2})+ O(p_N^{-\rho}),\\
\|\hat{u}-u_{w,0}\|_{L_2}&\leq& O((Np_N)^{-1/2})+ O(p_N^{-\rho}).
\end{eqnarray*}
If we take $p_N=O(N^{-1/(1-2\rho)})$, then the optimal rates of estimators are
\begin{eqnarray*}
\|\hat{\beta}-\beta_0\|_{L_2}&\leq& O(N^{\rho/(1-2\rho)}),\\
\|\hat{u}-u_{w,0}\|_{L_2}&\leq& O(N^{\rho/(1-2\rho)}).
\end{eqnarray*}
Furthermore, for each {\rm (S1), (S2)} or {\rm (S3)}, if $(Np_N)^{1/2}p_N^{-\rho}\rightarrow 0$, as $N\rightarrow \infty$,
 \begin{eqnarray*}
\sqrt{Np_N} \left[
 \begin{array}{c}
 \hat{\beta}-\beta_0\\
\hat{u}-u_{w,0}
 \end{array}
 \right]
 \stackrel{D}{\longrightarrow} N_{2p}(0, \Sigma_{\beta,u}^{-1}).
 \end{eqnarray*}
\end{theorem}

Theorem \ref{ParametricPart} provides the optimal result for the parametric part of the estimator for both the shape and scale functions.
This result is helpful for gaining deeper insight into the parametric part.
The condition $(N p_N)^{-1/(2m)} K\rightarrow \infty$ implies that the bias from the penalized spline is of a negligible order compared to the bias from the GPD approximation. 
However, this condition is not restrictive, as we can choose an arbitrarily large $K$ as long as it does not violate (C5).
The optimal rate $O(N^{\rho/(1-2\rho)})$ is similar to the minimax optimals (Drees 2001). 
In the last assertion of Theorem \ref{ParametricPart}, the condition $(Np_N)^{1/2}p_N^{-\rho}\rightarrow 0$ means that the standard deviation of the estimator of parametric part dominates the bias appeared from second-order condition of EVT. 
If we choose large $w_N$, such situation can be obtained. 
Even when the optimal $w_N$ is chosen, deriving the explicit form of the bias is difficult and has quite complicated expression since the bias from second-order condition of EVT contains not only the parametric part but also nonparametric part. 

\begin{remark}
\textup{
For the spline method, one must choose the number of knots $K$ and the smoothing parameters $\lambda_j$ and $\nu_j$ for $j=1,\ldots,d$.
Ruppert (2002) reported that the spline estimator has good performance with large number of knots, and the smoothing parameters chosen appropriately. 
In the R-packages {\sf mgcv} and {\sf evgam}, the automatic smoothing parameter selection is implemented for arbitrarily chosen $K$. 
Therefore, in all theorems, we state the asymptotic results using $K$ instead of $\lambda$ and $\nu$.  
}
\end{remark}

\begin{remark}
\textup{
As described herein, we established the asymptotic properties of the estimators with a deterministic threshold function. 
However, in practice, there is the problem of chooding the threshold function. 
For example, Wang and Tsai (2009) and Li et al. (2022) assumed $\tau(w\mid x,z)=w$. 
Chavez-Demoullin and Davison (2005), Youngman (2019), and Mhalla et al. (2019) used the specific threshold function from prior information of data. 
Consequently, in practice, the functional form of $\tau(\cdot \mid x,z)$ should be determined, but we cannot confirm whether such a function of $\tau$ is correct or not. 
Although the conditional quantile function $\tau(w\mid x,z)= q(1-1/w\mid x,z)$ seems to be natural (Beirlant et al. 2006, Daouia 2013), it remains unclear whether the same quantile level at each point $(X,Z)=(x,z)$ is reasonable, or not.  
As related to conditional quantile, one natural way to construct the fully data-driven threshold is to use the locally defined order statistics around $(X,Z)=(x,z)$. 
However, such a choice of threshold requires a large total sample size, which is unrealistic in many cases.
Thus, establishing the asymptotic result of GAM-GPD with fully data-driven threshold function is quite challenging. 
}
\end{remark}

\section{Conclusion}

We have developed the asymptotic theory for regression with GPD. 
The shape and scale parameters (functions of covariates) included in GPD were modeled using GAM. 
The estimator of each component was constructed using penalized $B$-spline method to the extreme value data selected using the peak over threshold method.
We demonstrated the $L_2$ and $L_\infty$ rate of convergence of the estimators of the shape and scale functions. 
We also presented the asymptotic normality of the estimator as the local asymptotics of the additive estimator.

The motivation for this study derived from an earlier study for which GPD regression with GAM was available in the R package {\sf evgam} provided by Youngman (2022). 
In the {\sf R}-package {\sf evgam}, GAM can also be used for the generalized extreme value (GEV) distribution instead of GPD. 
Development of theoretical results for GAM of GEV regression are anticipated as an important topic for further research. 
Such a topic can be regarded as an extension of work by Smith (1985) and B${\rm\ddot{u}}$cher and Segers (2017), which developed the asymptotic theory for maximum-likelihood estimator of GEV with a non-regression case.

\begin{appendices}

\section{Outline of proof of theorems}\label{secA1}

Our purpose is to evaluate the asymptotic behavior of the estimator $\hat{\gamma}(x,z)$ and $\hat{\sigma}(x,z)$ of the true additive model $\gamma_0(x,z)$ and $\sigma_{w0}(x,z)$. 
Let 
$$
L(\theta)=-E[\log h(Y|\bar{\gamma}(X,Z), \bar{\sigma}(X,Z))]
$$
and let 
$$
\theta_0=(\beta_0^\top , b_0^\top , u_{w,0}^\top , c_0^\top )^\top=\argmin_{(\beta,b,u,c)}L(\theta).
$$
Under (C1), $\theta_0$ is unique. 
We write $b_0=(b_{01}^\top ,\ldots,b_{0d}^\top )^\top $ and $b_{0j}=(b_{0,j,1},\ldots,b_{0,j,K+\xi})^\top \in\mathbb{R}^{K+\xi}$. 
Similarly, we write $c_0=(c_{01}^\top ,\ldots,c_{0d}^\top )^\top $ and $c_{0j}=(c_{0,j,1},\ldots,b_{c,j,K+\xi})^\top \in\mathbb{R}^{K+\xi}$. 
Then, for $(x,z)\in{\cal X}\times{\cal Z}$ with $z=(z^{(1)},\ldots,z^{(d)})^\top $, we define $\bar{g}_{0j}(z^{(j)})=B_j(z^{(j)})^\top b_{0j}$ and $\bar{s}_0(z^{(j)})=B_j(z^{(j)})^\top c_{0j}$ for $j=1,\ldots,d$. 
Therefore, the best spline approximation of $(\gamma_0,\log\sigma_{w0})$ can be written as
\[
\bar{\gamma}_0(x,z)=x^\top \beta_0 +\sum_{j=1}^d \bar{g}_{0j}(z^{(j)}),
\]
and 
\[
\log \bar{\sigma}_{w0}(x, z)= x^\top u_{w,0} +\sum_{j=1}^d \bar{s}_0(z^{(j)}),
\] 
respectively.
From Lemma \ref{splineap} in Appendix B, we see that $\bar{g}_{0j}(z^{(j)})-g_{0j}(z^{(j)})=O(K^{-\zeta})$ and $\bar{s}_{0j}(z^{(j)})-s_{0j}(z^{(j)})=O(K^{-\zeta})$. 
Thus, we obtain
\begin{eqnarray*}
\hat{g}_j(z^{(j)})-g_{0j}(z^{(j)})&=&\hat{g}_j(z^{(j)})-\bar{g}_{0j}(z^{(j)})+\bar{g}_{0j}(z^{(j)})-g_{0j}(z^{(j)})\\
&=&B_j(z^{(j)})^\top (\hat{b}-b_0)+O(K^{-\zeta})
\end{eqnarray*}
and 
\begin{eqnarray*}
\hat{s}_j(z^{(j)})-s_{0j}(z^{(j)})&=&\hat{s}_j(z^{(j)})-\bar{s}_{0j}(z^{(j)})+\bar{s}_{0j}(z^{(j)})-s_{0j}(z^{(j)})\\
&=&B_j(z^{(j)})^\top (\hat{c}-c_0)+O(K^{-\zeta}).
\end{eqnarray*}
Accordingly, we have
\begin{eqnarray*}
\hat{\gamma}(x,z)-\gamma_0(x,z)&=&\hat{\gamma}(x,z)-\bar{\gamma}_0(x,z)+\bar{\gamma}_0(x,z)-\gamma_0(x,z)\\
&=& x^\top (\hat{\beta}-\beta_0)+ B(z)^\top (\hat{b}-b_0)+O(K^{-\zeta})
\end{eqnarray*}
and 
\begin{eqnarray*}
\log\hat{\sigma}(x,z)-\log\sigma_{w0}(x,z)&=&\log\hat{\sigma}(x,z)-\log\bar{\sigma}_{w0}(x,z)+\log\bar{\sigma}_{w0}(x,z)-\log\sigma_{w0}(x,z)\\
&=& x^\top (\hat{u}-u_{w0})+ B(z)^\top (\hat{c}-c_0)+O(K^{-\zeta}).
\end{eqnarray*}
Thus, we aim to analyze the asymptotic behavior of $\hat{\theta}-\theta_0$. 
We write $\ell_{pen}(\theta)=\ell_{pen}(\beta,b,u,c)$. 
If $\|\hat{\theta}-\theta_0\|\stackrel{P}{\to}0$, it guarantees to use the Taylor expansion to $\ell_{pen}(\theta)$ as 
\begin{eqnarray}
\left(\frac{\partial^2 \ell_{pen}(\theta_0)}{\partial\theta\partial\theta^\top }\right)(\hat{\theta}-\theta_0)&=& 
\left\{\frac{\partial \ell_{pen}(\theta_0)}{\partial\theta}-E\left[\frac{\partial \ell_{pen}(\theta_0)}{\partial\theta}\right]\right\}(1+o_P(1))\nonumber\\
&&+E\left[\frac{\partial \ell_{pen}(\theta_0)}{\partial\theta}\right](1+o_P(1)) \label{Tay}.
\end{eqnarray}
As our result, the asymptotic orders of first and second term of the right hand side of (\ref{Tay}) are presented in Lemmas \ref{GradientSt} and \ref{GradientEx} whereas Lemma \ref{Hessian} proves the boundness of $(\partial^2 \ell_{pen}(\theta_0)/\partial \theta\partial\theta^\top )$. 
After showing $\|\hat{\theta}-\theta_0\|\stackrel{P}{\to}0$ in Lemma \ref{ConvPara}, Lemmas \ref{RatePara} and \ref{AS.para} derive the rate of convergence of $\|\hat{\theta}-\theta_0\|$ and asymptotic normality of $\hat{\theta}-\theta_0$ using (\ref{Tay}). 
To achive these lemmas, we use the preliminary results given in Lemmas \ref{bernstein}--\ref{penalty}. 
All lemmas and their proofs are provided in Appendix B.
The proofs of the theorems in Section 3 are given in Appendix C.

\section{Lemmas}

\begin{lemma}[Bernstein's inequality]\label{bernstein}
Let $W_1,\ldots,W_n$ be univariate $i.i.d.$ random variable having $E[W_i]=0$, $V[W_i]=\sigma_i^2\in(0,\infty)$ and $|W_i|<M$ almost surely for some constant $M>0$. 
Then, for any $\varepsilon>0$, 
\begin{eqnarray*}
P\left(\sum_{i=1}^n W_i >\varepsilon\right)\leq \exp\left[\frac{-2^{-1}\varepsilon^2}{\sum_{i=1}^n \sigma_i^2 + 3^{-1}\varepsilon M }\right].
\end{eqnarray*}
\end{lemma}

Lemma \ref{bernstein} is the famous result in probability theory. 
Its proof is clarified in van der Geer (2000). 

\begin{lemma}\label{Bsplinematrix}
There exist constants $M_{min}, M_{max}>0$ such that for any nonzero-vector $v\in\mathbb{R}^{d(K+\xi)}$,
$$
M_{min}\leq \frac{v^\top E[B(Z)B(Z)^\top ]v }{\|v\|^2}\leq M_{max}.
$$ 
\end{lemma}

Lemma \ref{Bsplinematrix} means that $E[B(Z)B(Z)^\top ]$ is positive definite matrix. 
The proof of Lemma \ref{Bsplinematrix} is shown in Lemma A.2 of Liu et al. (2011).

\begin{lemma}\label{splineap}
Suppose that {\rm (C2)}. 
Then,  as $K\rightarrow\infty$, 
$$
\sup_{z^{(j)}\in[0,1]} \left|B_j(z^{(j)})^\top b_{0,j}- g_{0j}(z^{(j)})\right|=O(K^{-\zeta}),\ \ j=1,\ldots,d,
$$
and 
$$
\sup_{z^{(j)}\in[0,1]} \left|B_j(z^{(j)})^\top c_{0,j}- s_{0j}(z^{(j)})\right|=O(K^{-\zeta}),\ \ j=1,\ldots,d.
$$
Consequently, 
$$
\sup_{(x,z)\in{\cal X}\times{\cal Z}} |\bar{\gamma}(x,z)-\gamma_0(x,z)|=O(K^{-\zeta})
$$
and 
$$
\sup_{(x,z)\in{\cal X}\times{\cal Z}} |\log \bar{\sigma}(x,z)-\log \sigma_{w0}(x,z)|=O(K^{-\zeta})
$$
\end{lemma}

\begin{proof}[Proof of Lemma \ref{splineap}]

From de Boor (2001), for $g_{0j}\in{\cal C}^{(\zeta)}$, there exists $b_j^*=(b_{j,1}^*,\ldots,b_{j,K+\xi}^*)^\top \in\mathbb{R}^{K+\xi}$ such that 
$$
\sup_{z\in[0,1]} \left|\sum_{k=1}^{K+\xi}B_k^{[\xi]}(z^{(j)})b_{j,k}^*-g_{0j}(z)\right|=O(K^{-\zeta}).
$$
Similarly, for $s_{0j}\in{\cal C}^{(\zeta)}$, there exists $c_j^*=(c_{j,1}^*,\ldots,c_{j,K+\xi}^*)^\top \in\mathbb{R}^{K+\xi}$ such that 
$$
\sup_{z\in[0,1]} \left|\sum_{k=1}^{K+\xi}B_k^{[\xi]}(z^{(j)})c_{j,k}^*-s_{0j}(z)\right|=O(K^{-\zeta}).
$$
Define $b^*=((b_1^*)^\top,\ldots,(b_d^*)^\top)$,  $c^*=((c_1^*)^\top,\ldots,(c_d^*)^\top)$ and $\theta^*=(\beta_0^\top, (b^*)^\top, u_{w0}^\top, (c^*)^\top)^\top$. 
We now assume that Lemma \ref{splineap} does not hold, that is, $K^\zeta|B_j(z^{(j)})^\top b_{0,j}- g_{0j}(z^{(j)})|\rightarrow \infty$ and $K^\zeta|B_j(z^{(j)})^\top c_{0,j}- s_{0j}(z^{(j)})|\rightarrow \infty$. 
Then, from (\ref{TrueRisk}) and the convexity of $L(\theta)$,
\[
-E[\log h(Y\mid \gamma_0(X,Z), \sigma_{w,0}(X,Z))] \leq L(\theta^*) <L(\theta_0)
\]
must be satisfied. 
However, this contradicts to the minimality of $\theta_0$ on $L(\theta)$. 
This completes the proof. 
\end{proof}

We next show the quadratic form and asymptotic order of penalty term in (\ref{Lpen}). 
Let $\Psi_j=(\Psi_{j, i,k})_{ik}$ be $(K+\xi+1)\times (K+\xi)$matrix with $\Psi_{j,i,i}=1/\|\bar{\psi}^{[\xi]}_{j,i}\|_2$, $\Psi_{j, i+1,i}=-(\phi_{j, i}/\phi_{j,i-1})/\|\bar{\psi}^{[\xi]}_{j,i}\|_2$ for $i=1,\ldots,K+\xi$ and $\Psi_{j,i,k}=0$ for $|i-k|\geq 2$. 
Then, we can write $\bar{g}_j(z_j)=B(z_j)^\top  b_j= \psi^{[\xi]}(z_j)^\top  \Psi_j b_j$, where $\psi^{[\xi]}(z_j)=(\psi_0^{[\xi]}(z_j), \ldots, \psi_{K+\xi}^{[\xi]}(z_j))^\top $ is the vector of original $B$-spline bases vector. 
Let $D_{1,K}=(D_{1,K,i,j})_{ij}$ be $(K+\xi)\times (K+\xi+1)$ matrix with $D_{1,K,i,i}=1$ and $D_{1,K,i,i+1}=-1$ for $i=1,\ldots,K+\xi$, and $D_{1,i,j}=0$ for $|i-j|>2$. 
Then, $D_{1,K}$ is the so called first order difference matrix (see, Xiao 2019).
For $q\geq 2$, let $D_{q,K}$ be $(K+\xi+1-q)\times(K+\xi+1)$ matrix satisfying $D_{q,K}=D_{1,K+1-q}D_{q-1,K}$ recursively. 
Then, $D_{q,K}$ is denoted by $q$th order difference matrix. 
Lastly, we define $(K+\xi+1-m)$ matrix $R_m=(R_{i,k})_{ik}$ with $R_{i,k}=\int_0^1 \psi^{[\xi-m]}_{i-1}(z)\psi^{[\xi-m]}_{k-1}(z)dz$ for $i,k=1,\ldots,K+\xi+1-m$. 
By using the derivative of $\psi_k^{[\xi]}(z_j)$, we can evaluate the quadratic form of the penalty term.

\begin{lemma}\label{penalty}

For any $v\in\mathbb{R}^{K+\xi}$, 
$$
\frac{\partial^2}{\partial v\partial v^\top }\int_0^1 \left\{\frac{d ^m B_j(z)^\top  v}{d z^m} \right\}^2 dz= K^{2m} (m!)^2 \Psi_j^\top  D_{m,K}^\top  R_m D_{m,K} \Psi_j=O(K^{2m}), j=1,\ldots,d.
$$
\end{lemma}

\begin{proof}[Proof of Lemma \ref{penalty}]
By the definition of normalized $B$-spline, we obtain $B_j(z)^\top v=\psi^{[\xi]}(z)^\top  \Psi_j v$ for $z\in[0,1]$. 
Furthermore, from the property of $m$th derivative of $B$-spline function (see, de Boor 2001, Xiao 2019), we obtain that for $z\in[0,1]$,
$$
\frac{d^m}{d z^m}\psi^{[\xi]}(z)^\top  \Psi_j v = K^m m! \psi^{[\xi-m]}(z)^\top D_{m,K} \Psi_j v,
$$
where $\psi^{[\xi-m]}(z)=(\psi^{[\xi-m]}_{0}(z_j),\ldots, \psi^{[\xi-m]}_{K+\xi-m}(z_j))^\top $.
Accordingly, 
$$
\int_0^1 \left\{\frac{d ^m B_j(z)^\top  v}{d z^m} \right\}^2 dz= K^{2m} (m!)^2 v^\top \Psi_j^\top  D_m^\top  R_m D_m \Psi_jv
$$
Here, for $R_m=(R_{i,k})_{ik}$, we have $R_{i,k}=O(K^{-1})$ for $|i-k|\leq \xi-m$ and $R_{i,k}=0$ otherwise, that is, $R_m$ is the band matrix. 
Furthermore, for $\Psi_j=(\Psi_{j, i,k})_{ik}$, $\Psi_{j, i,k}=O(K^{1/2})$ for $k=i+1$ and $\Psi_{j, i,k}=0$ otherwise. 
Thus, $\Psi_j$ is also band matrix. 
This implies that $K^{2m} (m!)^2\Psi_j^\top  D_m^\top  R_m D_m \Psi_j=O(K^{2m})$. 
\end{proof}
%%%%%%%%%%

Define $\theta_\gamma=(\beta^\top , b^\top )^\top $ and $\theta_\sigma=  (u^\top ,c^\top )^\top $. 
From these symbols, we have $\bar{\gamma}(X,Z)=A(X,Z)^\top \theta_\gamma$ and $\bar{\sigma}(X,Z)=\exp[A(X,Z)^\top \theta_\sigma]$. 
We further define some symbols. 
Let
\begin{eqnarray}
&&\ell_\gamma(y|\gamma(x,z),\sigma(x,z))\nonumber\\
&&=
\left.\frac{\partial}{\partial a}\left\{ -\log h(y|a,b)\right\}\right|_{a=\gamma(x,z),b=\sigma(x,z)}\nonumber\\
&&=
(\gamma(x,z)^{-1}+1)\frac{y/\sigma(x,z)}{1+y\gamma(x,z)/\sigma(x,z)}-\gamma(x,z)^{-2}\log\left(1+\frac{y\gamma(x,z)}{\sigma(x,z)}\right) \label{gradientGamma}
\end{eqnarray}
and 
\begin{eqnarray}
\ell_\sigma(y|\gamma(x,z),\sigma(x,z))
&=&
\left.\frac{\partial}{\partial \log b}\left\{ -\log h(y|a,b)\right\}\right|_{a=\gamma(x,z),b=\sigma(x,z)}\nonumber\\
&=&
1-(\gamma(x,z)^{-1}+1)\frac{y \gamma(x,z)/\sigma(x,z)}{1+y \gamma(x,z)/\sigma(x,z) }. \label{gradientSigma}
\end{eqnarray}
If $\gamma(x,z)=0$, 
\begin{eqnarray}
\ell_\gamma(y|0,\sigma(x,z))=\frac{y}{\sigma(x,z)}-\frac{1}{2}\frac{y^2}{\sigma(x,z)} \label{gradientZero}
\end{eqnarray}
and $\ell_\sigma(y|0,\sigma(x,z))=1-y/\sigma(x,z)$. 
We can show that 
\begin{eqnarray*}
&&\int \ell_\gamma(y|\gamma_0(x,z),\sigma_{w0}(x,z)) dH(y|\gamma_0(x,z),\sigma_{w0}(x,z))\\
&&=\int \ell_\sigma(y|\gamma_0(x,z),\sigma_{w0}(x,z)) dH(y|\gamma_0(x,z),\sigma_{w0}(x,z))\\
&&=0
\end{eqnarray*}
from the property of gradient of log-likelihood function of GPD. 
Furthermore, the Fisher information matrix of log-likelihood of GPD involves 
\begin{eqnarray*}
\int \ell_\gamma(y|\gamma_0(x,z),\sigma_{w0}(x,z))^2 dH(y|\gamma_0(x,z),\sigma_{w0}(x,z))&=&
\frac{2}{(2\gamma_0(x,z)+1)(\gamma_0(x,z)+1)},\\
\int \ell_\sigma(y|\gamma_0(x,z),\sigma_{w0}(x,z))^2 dH(y|\gamma_0(x,z),\sigma_{w0}(x,z))&=&
\frac{1}{2\gamma_0(x,z)+1}
\end{eqnarray*}
and 
\begin{eqnarray*}
&&\int \ell_\gamma(y|\gamma_0(x,z),\sigma_{w0}(x,z))\ell_\sigma(y|\gamma_0(x,z),\sigma_{w0}(x,z)) dH(y|\gamma_0(x,z),\sigma_{w0}(x,z))\\
&&=\frac{1}{(2\gamma_0(x,z)+1)(\gamma_0(x,z)+1)}.
\end{eqnarray*}

These results are used in the proof of Lemmas \ref{GradientSt}, \ref{GradientEx} and \ref{Hessian} below.

\begin{lemma}\label{GradientSt}
Suppose that {\rm (C1)--(C5)}. 
In each scenario {\rm (S1), (S2) or (S3)}, as $N\rightarrow\infty$, 
\[
E\left[\left\| \frac{\partial }{\partial \theta} \ell_{pen}(\theta_0)-  E\left[ \frac{\partial }{\partial \theta} \ell_{pen}(\theta_0)\right]\right\|^2\right] =O(K p_N/N).
\]
\end{lemma}

\begin{proof}[Proof of Lemma \ref{GradientSt}]
Since the penalty term of $\ell_{pen}$ does not contain the stochastic structure, we have 
\[
\left\| \frac{\partial }{\partial \theta} \ell_{pen}(\theta_0)-  E\left[ \frac{\partial }{\partial \theta} \ell_{pen}(\theta_0)\right]\right\|^2
=\left\| \frac{\partial }{\partial \theta} \ell(\theta_0)-  E\left[ \frac{\partial }{\partial \theta} \ell(\theta_0)\right]\right\|^2.
\]
We let $\bar{\ell}_A= \ell_A(Y|\bar{\gamma}_0(X,Z),\bar{\sigma}_{w0}(X,Z))$ for $A=\{\gamma,\sigma\}$ and let $I$ be $(p+K+\xi)$-identity matrix. 
We write $C=(X^\top , B(Z)^\top , X^\top , B(Z)^\top )^\top$.
Then, from Lemma \ref{splineap}, we obtain
\begin{eqnarray*}
&&E\left[
\left\| \frac{\partial }{\partial \theta} \ell(\theta_0)-  E\left[ \frac{\partial }{\partial \theta} \ell(\theta_0)\right]\right\|^2
\right]\\
&&=\frac{1}{N}
E\left[P(Y>0\mid X,Z)
\left. C^\top
\left(
\begin{array}{cc}
\bar{\ell}_\gamma^2I&\bar{\ell}_\gamma\bar{\ell}_\sigma I\\
\bar{\ell}_\sigma\bar{\ell}_\gamma I&\bar{\ell}_\sigma^2I
\end{array}
\right)
C
\right| Y>0
\right]\\
&&=
\frac{1}{N}E\left[P(Y>0\mid X,Z) \left.\frac{\gamma(X,Z)+5}{(2\gamma(X,Z)+1)(\gamma(X,Z)+1)}(\|X\|^2 + \|B(Z)\|^2)\right| Y>0\right]\\
&&\quad\quad\times (1+o(1)).
\end{eqnarray*}
From the mean value theorem for integrals, there exists $x^*\in\mathbb{\cal X}$ and $z^*\in{\cal Z}$ such that 
\begin{eqnarray*}
&&E\left[P(Y>0\mid X,Z) \left.\frac{\gamma(X,Z)+5}{(2\gamma(X,Z)+1)(\gamma(X,Z)+1)}(\|X\|^2 + \|B(Z)\|^2)\right| Y>0\right]\\
&&=
E[P(Y>0\mid X,Z) ]
\frac{\gamma(x^*,z^*)+5}{(2\gamma(x^*,z^*)+1)(\gamma(x^*,z^*)+1)}(\|x^*\|^2 + \|B(z^*)\|^2).
\end{eqnarray*}
By the property of normalized $B$-spline, we have $\|B(z^*)\|^2= O(K)$, which implies that 
\[
E\left[
\left\| \frac{\partial }{\partial \theta} \ell(\theta_0)-  E\left[ \frac{\partial }{\partial \theta} \ell(\theta_0)\right]\right\|^2
\right]
=O(K p_N/N).
\] 
\end{proof}

\begin{lemma}\label{GradientEx}
Suppose that {\rm (C1)--(C5)}. 
Then, in each scenario {\rm (S1), (S2) or (S3)}, as $N\rightarrow\infty$,
\[
\left\|E\left[ \frac{\partial }{\partial \theta} \ell_{pen}(\theta_0)\right]
\right\|^2\leq O(p_N^2 K^{-2m}) + O\left( p_N^{2(1-\rho)}\right).
\]
\end{lemma}

\begin{proof}[Proof of Lemma \ref{GradientEx}]
By the triangle inequality, we have 
\[
\left\|E\left[ \frac{\partial }{\partial \theta} \ell_{pen}(\theta_0)\right]
\right\|^2
\leq \left\|E\left[ \frac{\partial }{\partial \theta} \ell(\theta_0)\right]\right\|^2+\left\|\frac{\partial }{\partial \theta}\theta_0^\top \Omega_{\gamma,\sigma}\theta_0 
\right\|^2.
\]

We first consider the part of log-likelihood. 
Again, we use the symbol $\bar{\ell}_A, A\in\{\gamma,\sigma\}$ defined in the proof of Lemma \ref{GradientSt}. 
Then, we have 
\begin{eqnarray*}
E\left[ \frac{\partial }{\partial \theta} \ell(\theta_0)\right]
=
E\left[
\begin{array}{l}
P(Y>0)E[\bar{\ell}_\gamma\mid Y>0,X,Z] X\\
P(Y>0)E[\bar{\ell}_\gamma\mid Y>0,X,Z] B(Z)\\
P(Y>0)E[\bar{\ell}_\sigma\mid Y>0,X,Z] X\\
P(Y>0)E[\bar{\ell}_\sigma\mid Y>0,X,Z] B(Z)
\end{array}
\right] .
\end{eqnarray*}
In following, for simplicity, we write $\alpha(w_N|x,z)$ as $\alpha(\tau(w_N|x,z)|x,z)$.
For any integrable function $q(Y)$, we have that for any $(x,z)\in{\cal X}\times{\cal Z}$,
$$
E_{Y|x,z}[q(Y)\mid Y>0]=\int q(y)dH(y|x,z)+ \alpha(w_N|x,z)\int q(y)Q^\prime(y/\sigma_{w0}(x,z)|x,z)dy(1+o(1)),
$$
where, $E_{Y|x,z}$ is the expectation by the conditional distribution of $Y_i$ given $(X_i,Z_i)=(x,z)$ and 
$H(y|x,z)=H(y|\gamma_0(x,z),\sigma_{w0}(x,z))$.
If $q(y)=\bar{\ell}_A, A\in\{\gamma,\sigma\}$, we have $\int q(y)dH(y|x,z)=0$.
Since $\bar{\gamma}_0(x,z) -\gamma_0(x,z)=O(K^{-\zeta})$ and $\bar{\sigma}_{w0}(x,z)-\sigma_{w0}(x,z)=O(K^{-\zeta})$ from Lemma \ref{splineap}, we obtain 
\begin{eqnarray*}
&&E_{Y|x,z}\left[\bar{\ell}_\gamma \mid Y>0\right]\\
&&
=
E_{Y|x,z}\left[
\ell_\gamma(y|\bar{\gamma}_0(x,z),\bar{\sigma}_{w0}(x,z))
\mid Y>0\right]\\
&&
=\alpha(\tau(w_N|x,z)|x,z)\int \ell_\gamma(y|\gamma_0(x,z),\sigma_{w0}(x,z)) Q^\prime(y/\sigma_{w0}(x,z)|x,z)dy(1+o(1))+O(K^{-\zeta})\\
&&\equiv \alpha(\tau(w_N|x,z)|x,z) q_{\gamma}(x,z)+O(K^{-\zeta}).
\end{eqnarray*}
By the definition of $Q$, we can find that $|q_\gamma(x,z)|<\infty$ for all $(x,z)\in{\cal X}\times{\cal Z}$ under each case (S1)--(S3). 
Similary, we have 
$$
E_{Y|x,z}\left[\bar{\ell}_\sigma\mid Y>0\right]
= \alpha(\tau(w_N|x,z)|x,z)q_{\sigma}(x,z)(1+o(1))+O(K^{-\zeta}),
$$
where 
$$
q_{\sigma}(x,z)=\int  \ell_{\sigma}(y|\gamma_0(x,z),\sigma_{w0}(x,z))Q^\prime(y/\sigma_{w0}(x,z)|x,z)dy
$$
and $|q_{\sigma}(x,z)|<\infty$ for all $(x,z)\in{\cal X}\times{\cal Z}$.
Thus, 
\begin{eqnarray*}
E\left[ \frac{\partial }{\partial \theta} \ell(\theta_0)\right]
&=&
E\left[
\begin{array}{l}
P(Y>0\mid X,Z)\alpha(\tau(w_N|X,Z)|X,Z)q_{\gamma}(X,Z)\{1+O(K^{-\zeta})\}X\\
P(Y>0\mid X,Z)\alpha(\tau(w_N|X,Z)|X,Z)q_{\gamma}(X,Z)\{1+O(K^{-\zeta})\}B(Z)\\
P(Y>0\mid X,Z)\alpha(\tau(w_N|X,Z)|X,Z)q_{\sigma}(X,Z)\{1+O(K^{-\zeta})\} X\\
P(Y>0\mid X,Z)\alpha(\tau(w_N|X,Z)|X,Z)q_{\sigma}(X,Z)\{1+O(K^{-\zeta})\} B(Z)
\end{array}
\right].
\end{eqnarray*}
By the mean value theorem for integrals, there exists $(x^*, z^*)\in{\cal X}\times {\rm Z}$ such that 
\begin{eqnarray*}
&&E[P(Y>0\mid X,Z)\alpha(\tau(w_N|X,Z)|X,Z)q_{A}(X,Z)X(1+o(1))]\\
&&=E[P(Y>0\mid X,Z)\alpha(\tau(w_N|X,Z)|X,Z)] q_{A}(x^*,z^*)x^*(1+o(1))
\end{eqnarray*}
for $A\in\{\gamma,\sigma\}$. 
Under (C4), we obtain 
\begin{eqnarray*}
E[P(Y>0\mid X,Z)\alpha(\tau(w_N|x,z)|X,Z)]
&=& O\left(E\left[P(Y>0|X,Z)^{1-\rho(X,Z)}\right]\right)\\
&\leq& O(P(Y>0)^{1-\rho})\\
&=&O(p_N^{1-\rho}).
\end{eqnarray*}
Similarly, there exists $(x^{**}, z^{**})\in{\cal X}\times {\rm Z}$ such that 
\begin{eqnarray*}
&&E[P(Y>0\mid X,Z)\alpha(\tau(w_N|x,z)|X,Z)q_{A}(X,Z)B(Z)(1+o(1))]\\
&&=E[P(Y>0\mid X,Z)\alpha(\tau(w_N|x,z)|X,Z)] q_{A}(x^{**},z^{**})B(z^{**})(1+o(1))\\
&&\leq O(p_N^{1-\rho}).
\end{eqnarray*}
By the property of normalized $B$-spline, we can evaluate
$$
\left\|E\left[ \frac{\partial }{\partial \theta} \ell(\theta_0)\right]\right\|^2 \leq O\left(p_N^{2(1-\rho)}\right).
$$

We next derive the asymptotic order of the penalty term:
$$
\frac{\partial}{\partial \theta_0} \sum_{j=1}^d\left\{ \lambda_j\int_0^1 \{\bar{g}^{(m)}_j(z)\}^2dz
+\nu \int_0^1\{\bar{s}^{(m)}_j(z)\}^2dz\right\}.
$$
From Lemma \ref{penalty}, we have
$$
 \sum_{j=1}^d \lambda_j\int_0^1 \{\bar{g}^{(m)}_j(z)\}^2dz=\frac{K^{2m}}{2}b_0^\top  \Omega(\lambda) b_0,
$$
 where $\Omega$ is the $d(K+\xi)$ square matrix with
 \begin{eqnarray*}
\Omega(\lambda)=
(m!)^2\left[
\begin{array}{ccc}
\lambda_1\Psi_1^\top D_m^\top R_mD_m\Psi_1& & \\
&\ddots &\\
&&\lambda_d \Psi_d^\top D_m^\top R_mD_m\Psi_d 
\end{array}
\right].
\end{eqnarray*}
Here, all elements in the off diagonal block (the part of blank) of $\Omega$ are zero. 
Similarly, we can write 
$$
 \sum_{j=1}^d \nu_j \int_0^1 \{\bar{s}^{(m)}_j(z)\}^2dz=\frac{K^{2m}}{2}c_0^\top  \Omega(\nu) c_0.
$$
Therefore, using 
\begin{eqnarray*}
\Omega_{\gamma,\sigma}
=
\left[
\begin{array}{cccc}
O&&&\\
&K^{2m}\Omega(\lambda)&&\\
&&O&\\
&&&K^{2m}\Omega(\nu)
\end{array}
\right],
\end{eqnarray*}
where the all elements of $O$ and the off-diagonal block are zero,  
we obtain 
$$
\frac{\partial}{\partial \theta_0} \sum_{j=1}^d\left\{ \lambda_j\int_0^1 \{\bar{g}^{(m)}_j(z)\}^2dz
+\nu_j \int_0^1\{\bar{s}^{(m)}_j(z)\}^2dz\right\}=\Omega_{\gamma,\sigma}\theta_0.
$$

Meanwhile, from the definition of $R_m$, $D_m$, $\Psi_j$, $b_{0j}$ and Lemma \ref{penalty}, we obtain 
\begin{eqnarray*}
&&\lambda_j K^{2m}(m!)^2\Psi_j^\top D_m^\top R_mD_m\Psi_j b_{0j}\\
&&=\lambda_j K^m m! \Psi_j^\top D_m^\top 
\left[
\begin{array}{c}
\int_0^1 \psi_1^{[\xi-m]}(z)g_{0j}^{(m)}(z)dz\\
\vdots\\
\int_0^1 \psi_{K+\xi-m}^{[\xi-m]}(z)g_{0j}^{(m)}(z)dz
\end{array}
\right]  (1+o(1))\\
&&= O(\lambda_j K^{m})
\end{eqnarray*}
and 
\begin{eqnarray*}
\nu_j K^{2m}(m!)^2\Psi_j^\top D_m^\top R_mD_m\Psi_j c_{0j}= O(\nu_j K^{m} ).
\end{eqnarray*}
Therefore, we obtain $K^{2m}\Omega(\lambda) b_0=O(\max_j\lambda_j K^{m})$ and  $K^{2m} \Omega(\nu) c_0=O(\max _j \nu_j K^{m})$. 
Since $\Omega(\lambda)$ and $\Omega(\nu)$ are band matrix,  we have $\|K^{2m}\Omega(\lambda) b_0\|^2=O((\max_j \lambda_j)^2 K^{2m})$ and $\| K^{2m} \Omega(\nu) c_0\|=O((\max_j \nu_j)^2 K^{2m})$. 
This implies 
\[
\left\|\frac{\partial }{\partial \theta}\theta_0^\top \Omega_{\gamma,\sigma}\theta_0 
\right\|^2 \leq O(\lambda^2 K^{2m}) + O(\nu^2 K^{2m}).
\]
Under (C5), we obtain $O(\lambda^2 K^{2m}) + O(\nu^2 K^{2m})=O(p_N^2K^{-2m})$.
\end{proof}

\begin{lemma}\label{Hessian} 
Suppose that {\rm (C1)--(C5)}. 
In each scenario {\rm (S1), (S2) or (S3)}, 
for any vector $v\in\mathbb{R}^{2(p+K+\xi)}-\{0\}$ with $\|v\|= 1$, there exist positive constants $C_m$ and $C_M$ such that
\[
p_N C_m < v^\top E\left[\frac{\partial^2}{\partial \theta \partial \theta^\top } \ell_{pen}(\theta_0)\right] v < p_N C_M .
\]
\end{lemma}

\begin{proof}[Proof of Lemma \ref{Hessian}]
We first obtain
\begin{eqnarray*}
\frac{\partial^2}{\partial \theta \partial \theta^\top } \ell_{pen}(\theta)
=
\frac{\partial^2}{\partial \theta \partial \theta^\top } \ell(\theta)+ \Omega_{\gamma,\sigma},
\end{eqnarray*}
where $\Omega_{\gamma,\sigma}$ is that given in the proof of Lemma \ref{GradientEx}. 
Next, for the log-likelihood part, we obtain 
\[
E\left[
\frac{\partial^2}{\partial \theta \partial \theta^\top } \ell(\theta)
\right]
=
E_{X,Z}\left[P(Y>0\mid X, Z) E_{Y|x,z}\left[
\frac{\partial^2}{\partial \theta \partial \theta^\top } \ell(\theta)
\mid Y>0\right]\right].
\]
By the straightforward calculation of Fisher information matrix of log-likelihood of GPD, we have 
\begin{eqnarray*}
&&E_{Y|x,z}\left[
\frac{\partial^2}{\partial \theta \partial \theta^\top } \ell(\theta)
\mid Y>0\right]\\
&&=
\frac{1}{2\gamma_0(x,z)+1}\left[
\begin{array}{cc}
\frac{2}{\gamma_0(x,z)+1}A(x,z)A(x,z)^\top & \frac{1}{\gamma_0(x,z)+1}A(x,z)A(x,z)^\top \\
\frac{1}{\gamma_0(x,z)+1}A(x,z)A(x,z)^\top &A(x,z)A(x,z)^\top 
\end{array}
\right](1+o_P(1)).
\end{eqnarray*}
From (C5) and Lemma \ref{penalty}, we have $\Omega_{\gamma,\sigma}= O(\lambda K^{2m})+O(\nu K^{2m})=O(p_N)$.
Furthermore, the conditions (C1) and (S2) imply that Fisher information matrix of log-likelihood of GPD is positive definite. 
Therefore, the proof is completed if we can show that 
\[
E\left[\frac{\partial^2}{\partial \theta \partial \theta^\top } \ell(\theta_0)\right]
=
E_{X,Z}\left[P(Y>0|X,Z) E_{Y|x,z}\left[\frac{\partial^2}{\partial \theta \partial \theta^\top } \ell(\theta_0) \mid Y>0\right]
\right]=O(p_N).
\]
Let $v\in \mathbb{R}^{2(p+K+\xi)}$ with $\|v\|=1$. Then, this can be devided as $v=(v_\gamma^\top ,v_\sigma^\top )^\top $ with $v_\gamma,v_\sigma\in\mathbb{R}^{p+K+\xi}$ and $\|v_\gamma\|\leq 1, \|v_\sigma\|\leq 1$. 
Then, 
\begin{eqnarray*}
v^\top  E_{Y|x,z}\left[\frac{\partial^2}{\partial \theta \partial \theta^\top } \ell(\theta_0)\mid Y>0\right]v&=&
2q_1(x,z)\{A(x,z)^\top v_\gamma\}^2+2q_1(x,z)v_\gamma^\top A(x,z)A(x,z)^\top v_\sigma\\
&&+ q_2(x,z)\{A(x,z)^\top v_\sigma\}^2,
\end{eqnarray*}
where $q_1(x,z)=(\gamma_0(x,z)+1)/(2\gamma_0(x,z)+1)$ and $q_2(x,z)=1/(2\gamma_0(x,z)+1)$.
From mean value theoem for integrals, there exists $(x^*,z^*)\in{\cal X}\times {\cal Z}$ such that 
\begin{eqnarray}
&&E[ P(Y>0\mid X,Z) q_1(X,Z)\{A(X,Z)^\top v_\gamma\}^2] \nonumber\\
&&=E[ P(Y>0\mid X,Z) ] q_1(x^*,z^*)\{A(x^*,z^*)^\top v_\gamma\}^2. \label{qGamma}
\end{eqnarray}
From the property of $B$-spline basis, we obtain $\{A(x^*,z^*)^\top v_\gamma\}^2=O(1)$. 
Therefore, (\ref{qGamma}) has an order $O(E[ P(Y>0\mid X,Z) ] )=O(p_N)$. 
By the similar argument, we obtain 
\[
E[P(Y>0\mid X,Z)q_1(X,Z)v_\gamma^\top A(X,Z)A(X,Z)^\top v_\sigma]=O(p_N)
\]
and 
\[
E[P(Y>0\mid X,Z) q_2(X,Z)\{A(X,Z)^\top v_\sigma\}^2]=O(p_N),
\]
which completes the proof. 
\end{proof}

Next, we show the consistency of $\|\hat{\theta}-\theta_0\|$. 
However, the purpose of Lemma \ref{ConvPara} is not to derive the optimal rate of convergence, which is shown in Lemma \ref{RatePara}.

\begin{lemma}\label{ConvPara}
Suppose that {\rm (C1)--(C5)}. 
In each scenario {\rm (S1), (S2) or (S3)}, for any $\varepsilon>0$,
\begin{eqnarray*}
P\left(\|\hat{\theta}-\theta_0\|>\varepsilon\right)\rightarrow 0,\ \ {\rm as}\ \ N\rightarrow\infty.
\end{eqnarray*}
\end{lemma}

\begin{proof}[Proof of Lemma \ref{ConvPara}]

From Hijort and Pollard (2011), we see that for $\varepsilon>0$, 
\begin{eqnarray*}
P\left(\|\hat{\theta}-\theta_0\|>\varepsilon\right)
\leq P\left(\sup_{\|\theta-\theta_0\|\leq \varepsilon}|\ell_{pen}(\theta)-L(\theta)| \geq 2^{-1}\inf_{\|\theta-\theta_0\|=\varepsilon}|L(\theta)-L(\theta_0)| \right).
\end{eqnarray*}
Thus, our purpose is to show  
$$
P\left(\sup_{\|\theta-\theta_0\|\leq \varepsilon}|\ell_{pen}(\theta)-L(\theta)| \geq 2^{-1}\inf_{\|\theta-\theta_0\|=\varepsilon}|L(\theta)-L(\theta_0)| \right)\rightarrow 0,\ \ {\rm as}\ \ N\rightarrow\infty.
$$
We first evaluate the lower bound of $\inf_{\|\theta-\theta_0\|=\varepsilon}|L(\theta)-L(\theta_0)| $.
Under $\|\theta-\theta_0\|=\varepsilon$, we can write $\theta=\theta_0+ \varepsilon \eta$, where $\eta$ is $2(p+K+\xi)$-vector satisfying $\|\eta\|=1$. 
Then, we have 
\begin{eqnarray*}
L(\theta)-L(\theta_0)
&=&L^\prime(\theta_0)(\theta-\theta_0)+\frac{1}{2}(\theta-\theta_0)^\top L^{\prime\prime}(\theta^*)(\theta-\theta_0)\\
&=&\varepsilon L^\prime(\theta_0)\eta+\frac{\varepsilon^2}{2} \eta^\top L^{\prime\prime}(\theta_0)\eta(1+o(1)) ,
\end{eqnarray*}
where $L^\prime(\theta)=\partial L(\theta)/\partial \theta$ and $L^{\prime\prime}(\theta)=\partial^2 L(\theta)/\partial \theta\partial \theta^\top $. 
From the definition of $\theta_0$, we have $L^\prime(\theta_0)=0$. 
On the other hand,  by the similar argument as the proof of Lemma \ref{GradientSt},$L^{\prime\prime}$ can be written as
\begin{eqnarray*}
&&L^{\prime\prime}(\theta_0)\\
&&=
E\left[
\begin{array}{cc}
\frac{2P(Y>0|X,Z)}{(2\gamma_0(X,Z)+1)(\gamma_0(X,Z)+1)}A(X,Z)A(X,Z)^\top & \frac{P(Y>0|X,Z)}{(2\gamma_0(X,Z)+1)(\gamma_0(X,Z)+1)} A(X,Z)A(X,Z)^\top \\
\frac{P(Y>0|X,Z)}{(2\gamma_0(X,Z)+1)(\gamma_0(X,Z)+1)} A(X,Z)A(X,Z)^\top &\frac{P(Y>0|X,Z)}{2\gamma_0(X,Z)+1} A(X,Z)A(X,Z)^\top 
\end{array}
\right]\\
&&\quad \times(1+o(1)).
\end{eqnarray*}
Since $L^{\prime\prime}(\theta_0)$ is positive definite, from Lemma \ref{Hessian}, there exist constants $C>0$ such that 
\begin{eqnarray*}
L(\theta)-L(\theta_0)
=\frac{\varepsilon^2}{2} \eta^\top L^{\prime\prime}(\theta_0)\eta(1+o(1))
\geq  C p_N \varepsilon^2.
\end{eqnarray*}
We replace $C\varepsilon^2$ with $\varepsilon^2$.
Then, we obtain
\begin{eqnarray*}
&&P\left(\sup_{\|\theta-\theta_0\|\leq \varepsilon}|\ell_{pen}(\theta)-L(\theta)| \geq 2^{-1}\inf_{\|\theta-\theta_0\|=\varepsilon}|L(\theta)-L(\theta_0)| \right)\\
&&\leq P\left(\sup_{\|\theta-\theta_0\|\leq \varepsilon}|\ell_{pen}(\theta)-L(\theta)| \geq  p_N \varepsilon^2\right)
\end{eqnarray*}

Next, since 
\begin{eqnarray*}
&&\sup_{\|\theta-\theta_0\|\leq \varepsilon}|\ell_{pen}(\theta)-L(\theta)|\\
&& \leq |\ell_{pen}(\theta_0)-L(\theta_0)|+ \sup_{\|\theta-\theta_0\|\leq \varepsilon}|\ell_{pen}(\theta)-\ell_{pen}(\theta_0)-L(\theta)+L(\theta_0)|,
\end{eqnarray*}
we have 
\begin{eqnarray*}
&&P\left(\sup_{\|\theta-\theta_0\|\leq \varepsilon}|\ell_{pen}(\theta)-L(\theta)| \geq  p_N\varepsilon^2 \right)\\
&&\leq  P\left(|\ell_{pen}(\theta_0)-L(\theta_0)|\geq 2^{-1} p_N\varepsilon^2 \right)\\
&&\quad +  P\left(\sup_{\|\theta-\theta_0\|\leq \varepsilon}|\ell_{pen}(\theta)-\ell_{pen}(\theta_0)-L(\theta)+L(\theta_0)| \geq 2^{-1} p_N\varepsilon^2\right)\\
&&\equiv {\cal J}_1+{\cal J}_2.
\end{eqnarray*}
We first show ${\cal J}_1\rightarrow 0$. 
By the definition of $\ell_{pen}$, we have
\begin{eqnarray*}
\ell_{pen}(\theta_0)-L(\theta_0)
&=&\ell_{pen}(\theta_0)-L(\theta_0)\\
&&+\sum_{j=1}^d\left\{ \lambda_j \int_0^1 \{\bar{g}^{(m)}_j(z)\}^2dz
+\nu_j \int_0^1\{\bar{s}^{(m)}_j(z)\}^2dz\right\}.
\end{eqnarray*}
From Lemma \ref{splineap}, we obtain
$$
\int_0^1 \{\bar{g}_j^{(m)}(z)\}^2dz=\int_0^1 \{g_j^{(m)}(z)\}^2dz=O(1)
$$
and 
$$
\int_0^1 \{\bar{s}_j^{(m)}(z)\}^2dz=\int_0^1 \{s_j^{(m)}(z)\}^2dz=O(1).
$$
Since $(\lambda+\nu)/p_N\rightarrow 0$ under (C5), we obtain 
$$
\frac{1}{p_N}\sum_{j=1}^d \left\{\lambda_j \int_0^1 \{\bar{g}^{(m)}_j(z)\}^2dz
+\nu_j \int_0^1\{\bar{s}^{(m)}_j(z)\}^2dz\right\} < \varepsilon^2.
$$
Therefore, to show ${\cal J}_1\rightarrow 0$, it is sufficient to prove 
$P(|\ell(\theta_0)-L(\theta_0)|>p_N\varepsilon^2)\rightarrow 0$. 
Note that
\begin{eqnarray*}
&&|\ell(\theta_0)-L(\theta_0)|\\
&&=\left|\frac{1}{N}\sum_{i=1}^N\log h(Y_i|\bar{\gamma}_0(X_i,Z_i),\bar{\sigma}_{w0}(X_i,Z_i))I(Y_i>0)\right.\\
&&\hspace{2mm}\left.-E[\log h(Y_i|\bar{\gamma}_0(X_i,Z_i),\bar{\sigma}_{w0}(X_i,Z_i))I(Y_i>0)]\right|
\end{eqnarray*}

Let $E_i=-\bar{\gamma}_0(X_i,Z_i)^{-1}\log h(Y_i|\bar{\gamma}_0(X_i,Z_i),\bar{\sigma}_{w0}(X_i,Z_i))$. 
Then, by the form of the density function of GPD, $E_i$ is approximately distributed as exponential distribution under $Y_i>0$. 
This implies that $V[E_i\mid Y_i>0]\leq C$ for some constant $C>0$. 
Therefore, by the Chebyshev's inequality, we obtain 
\[
{\cal J}_1\leq \frac{C}{Np_N \varepsilon^2}\rightarrow 0.
\]

Next, we show ${\cal J}_2\rightarrow 0$ as $N\rightarrow \infty$. 
For $\{\theta: \|\theta-\theta_0\|\leq \varepsilon\}$, we write $\theta=\theta_0+\varepsilon \eta$ with $\|\eta\|\leq 1$. 
From Taylor's theorem, we have
\begin{eqnarray*}
&&\ell_{pen}(\theta)-\ell_{pen}(\theta_0)\\
&&= \varepsilon\frac{\partial}{\partial \theta^\top}\ell(\theta_0) \eta +\frac{\varepsilon^2}{2p_N}\eta^\top  \left(\frac{1}{p_N}\frac{\partial^2 \ell(\theta_0)}{\partial\theta\partial\theta^\top }\right) \eta(1+o_P(1))\\
&&\quad +\sum_{j=1}^d \lambda_j\left[\int_0^1 \left\{\frac{d^m}{dx^m} B^{[\xi]}(x)^\top (b_{0j}+\varepsilon \eta_{\gamma, j}) \right\}^2 dx- \int_0^1 \left\{\frac{d^m}{dx^m} B^{[\xi]}(x)^\top b_{0j} \right\}^2 dx\right] \\
&&\quad\quad +\sum_{j=1}^d \nu_j\left[\int_0^1 \left\{\frac{d^m}{dx^m} B^{[\xi]}(x)^\top (c_{0j}+\varepsilon \eta_{\sigma j}) \right\}^2 dx-\int_0^1 \left\{\frac{d^m}{dx^m} B^{[\xi]}(x)^\top c_{0j} \right\}^2 dx \right],
\end{eqnarray*}
where $\eta_{\gamma,j}$ and $\eta_{\sigma,j}$ are the $(K+\xi)$-subvector of $\eta$ corresponding to $b_j$ and $c_j$ in $\theta$. 
Since $\|\eta\|\leq 1$, each element of $\eta_{\gamma,j}$ has an order $O(K^{-1/2})$. 
In addition, from the property of $B$-spline (see, de Boor 2001), we have $(d^mB^{[\xi]}(x)/dx^m)b_{0j}= g^{(m)}(x)(1+o(1))$ and $\int \|d^mB^{[\xi]}(x)/dx^m\|^2dx=O(K^{2m-1})$.
Combining these results of $B$-spline model, we have 
$$
\sum_{j=1}^d \lambda_j\left[\int_0^1 \left\{\frac{d^m}{dx^m} B^{[\xi]}(x)^\top (b_{0j}+\varepsilon \eta_{\gamma, j}) \right\}^2 dx- \int_0^1 \left\{\frac{d^m}{dx^m} B^{[\xi]}(x)^\top b_{0j} \right\}^2 dx\right]\leq O(\lambda K^{m}\varepsilon).
$$
Similarly, we have 
$$
\sum_{j=1}^d \nu_j\left[\int_0^1 \left\{\frac{d^m}{dx^m} B^{[\xi]}(x)^\top (c_{0j}+\varepsilon \eta_{\sigma j}) \right\}^2 dx-\int_0^1 \left\{\frac{d^m}{dx^m} B^{[\xi]}(x)^\top c_{0j} \right\}^2 dx \right]\leq O(\nu K^{m}\varepsilon).
$$
From (C5), we have $(\lambda+\nu )K^{m}\varepsilon/(p_N\varepsilon^2)=O(K^{-m}/\varepsilon)=o(1)$. 
That is, 
\begin{eqnarray*}
&&\ell_{pen}(\theta)-\ell_{pen}(\theta_0)\\
&&= \varepsilon\frac{\partial}{\partial \theta}\ell(\theta_0) \eta +\frac{\varepsilon^2}{2p_N}\eta^\top  \left(\frac{1}{p_N}\frac{\partial^2 \ell(\theta_0)}{\partial\theta\partial\theta^\top }\right) \eta(1+o_P(1))+o(1).
\end{eqnarray*}
Similarly, from the Taylor's theorem, we obtain
\[
L(\theta)-L(\theta_0)=\varepsilon\frac{\partial}{\partial \theta^\top}L(\theta_0) \eta+\frac{\varepsilon^2}{2}\eta^\top\left(\frac{1}{p_N} \frac{\partial^2 L(\theta_0)}{\partial\theta\partial\theta^\top }\right) \eta(1+o(1)).
\]
From these expansions, we obtain 
\begin{eqnarray*}
&&p_N^{-1}|\ell_{pen}(\theta)-\ell_{pen}(\theta_0)-L(\theta)+L(\theta_0)|\\
&&
=\varepsilon p_N^{-1}\left(\frac{\partial\ell(\theta_0)}{\partial\theta}-\frac{\partial L(\theta_0)}{\partial\theta}\right)^\top\eta\\
&&+\frac{\varepsilon^2}{2}\eta^\top\left(\frac{1}{p_N}\frac{\partial^2 \ell(\theta_0)}{\partial\theta\partial\theta^\top }-\frac{1}{p_N}\frac{\partial^2 L(\theta_0)}{\partial\theta\partial\theta^\top }\right)\eta(1+o_P(1)).
\end{eqnarray*}
Since $\partial^2 \ell(\theta)/\partial\theta\partial\theta^\top $ is the Hessian matrix of log-likelihood of $\theta$ and is continuous for $\theta$, it converges to its expectation. 
That is, we obtain
$$
\frac{1}{p_N}\frac{\partial^2 \ell(\theta_0)}{\partial\theta\partial\theta^\top }-\frac{1}{p_N}\frac{\partial^2 L(\theta_0)}{\partial\theta\partial\theta^\top } = o_P(1),
$$
which implies that
$$
\ell_{pen}(\theta)-\ell_{pen}(\theta_0)-L(\theta)+L(\theta_0)=\varepsilon\left(\frac{\partial\ell(\theta_0)}{\partial\theta}-\frac{\partial L(\theta_0)}{\partial\theta}\right)\eta +o(p_N\varepsilon^2).
$$
Thus, the remaining is to show
$$
P\left(\sup_{\|\eta\|\leq 1}\left|\varepsilon\left\{\frac{\partial}{\partial \theta}\ell(\theta_0)-E\left[\frac{\partial}{\partial \theta}\ell(\theta_0)\right]\right\}^\top \eta\right| > p_N\varepsilon^2\right)\rightarrow 0.
$$
For simplicity, we write 
\[
\ell_1(\theta_0)=\left\{\frac{\partial}{\partial \theta}\ell(\theta_0)-E\left[\frac{\partial}{\partial \theta}\ell(\theta_0)\right]\right\}.
\]
Then, from Cauchy–Schwarz inequality, we obtain 
\[
\sup_{\|\eta\|\leq 1}|\ell_1(\theta_0)^\top \eta|\leq \|\ell_1(\theta_0)\|,
\]
which indicates that 
\begin{eqnarray*}
P\left(\sup_{\|\eta\|\leq 1}\left|\ell_1(\theta_0)^\top \eta\right| > p_N\varepsilon\right)\leq 
P\left(\|\ell_1(\theta_0)\| > p_N\varepsilon\right)
=P\left(\|\ell_1(\theta_0)\|^2 > p_N^2\varepsilon^2\right).
\end{eqnarray*}
From Lemma \ref{GradientSt} and Markov's inequality, we have 
\[
P\left(\|\ell_1(\theta_0)\|^2 > p_N^2\varepsilon^2\right)
\leq \frac{E[\|\ell_1(\theta_0)\|^2]}{p_N^2\varepsilon^2}=O\left(\frac{1}{N p_N\varepsilon^2}\right)=o(1).
\]
Thus, ${\cal J}_2\rightarrow 0$ was proven, which completes the proof.
\end{proof}

\begin{lemma}\label{RatePara}
Suppose that {\rm (C1)--(C5)}.
In each scenario {\rm (S1), (S2) or (S3)}, as $N\rightarrow\infty$, 
\begin{eqnarray*}
E[\|\hat{\theta}-\theta_0\|^2]\leq O\left(\frac{K}{N p_N}\right)+O\left(K^{-2m}\right)+O\left(p_N^{-2\rho}\right).
\end{eqnarray*}
\end{lemma}

\begin{proof}[Proof of Lemma \ref{RatePara}]
From Lemma \ref{ConvPara}, we can use the Taylor expantion to $\partial \ell_{pen}(\theta)/\partial \theta$ around $\hat{\theta}=\theta_0$.
This indicates that 
$$
\left(\frac{\partial^2 \ell_{pen}(\theta^*)}{\partial \theta\partial\theta^\top }\right)(\hat{\theta}-\theta_0)
=\frac{\partial}{\partial \theta}\ell_{pen}(\theta_0),
$$
where $\theta^*$ is the parameter satisfying $\|\theta^*-\theta_0\|<\|\hat{\theta}-\theta_0\|$. 
From Lemma \ref{Hessian}, we obtain 
$$
\left\|\left(\frac{\partial^2 \ell_{pen}(\theta^*)}{\partial \theta\partial\theta^\top }\right)(\hat{\theta}-\theta_0)\right\|^2
\geq C_m^2  p_N^2\|\hat{\theta}-\theta_0\|^2(1+o_P(1)).
$$
Together with Lemmas \ref{GradientSt} and \ref{GradientEx}, we obtain 
$$
E\left[\left\|\frac{\partial}{\partial \theta}\ell_{pen}(\theta_0)\right\|^2\right]\leq  O_P\left(\frac{K p_N}{N}\right)+O\left(p_N^2 K^{-2m}\right)+O\left( p_N^{2(1-\rho)}\right).
$$
Thus, we obtain 
\begin{eqnarray*}
E[\|\hat{\theta}-\theta_0\|^2]&\leq& \frac{1}{C_m^2 p_N^2}\left\{O\left(\frac{K p_N}{N }\right)+O\left(p_N^2K^{-2m}\right)+O\left( p_N^{2(1-\rho)}\right)\right\}\\
&=&O\left(\frac{K}{N p_N }\right)+O\left( K^{-2m}\right)+O\left(p_N^{-2\rho}\right)
\end{eqnarray*}
\end{proof}

For two random variable $A_N$ and $B_N$, $A_N\stackrel{a.s}{\sim}B_N$ means that $A_N$ and $B_N$ have same distribution as $N\rightarrow\infty$. 

\begin{lemma}\label{AS.para}
Suppose that {\rm (C1)--(C5)}.
In each scenario {\rm (S1), (S2) or (S3)}, as $N\rightarrow\infty$, 
$$
\sqrt{\frac{N}{p_N}}\left\{\frac{\partial \ell(\theta_0)}{\partial \theta}-E\left[\frac{\partial \ell(\theta_0)}{\partial \theta}\right] \right\}\stackrel{a.s}{\sim} {\cal N}_{2(p+K+\xi)}(0, \Sigma).
$$

\end{lemma}

\begin{proof}[Proof of Lemma \ref{AS.para}]

In this proof, we denote $\ell_\gamma(Y_i|X_i,Z_i)=\ell_\gamma(Y_i|\bar{\gamma}_0(X_i,Z_i),\bar{\sigma}_{w0}(X_i,Z_i))$ and   $\ell_\sigma(Y_i|X_i,Z_i)=\ell_\sigma(Y_i|\bar{\gamma}_0(X_i,Z_i),\bar{\sigma}_{w0}(X_i,Z_i))$.
Similar to the proof of Lemma \ref{GradientEx}, we have 
\begin{eqnarray*}
\frac{\partial \ell(\theta_0)}{\partial \theta}
=
\frac{1}{N}\sum_{i=1}^N\left[
\begin{array}{l}
\ell_\gamma(Y_i|X_i,Z_i)A(X_i,Z_i)\\
\ell_\sigma(Y_i|X_i,Z_i)A(X_i,Z_i)
\end{array}
\right]I(Y_i>0).
\end{eqnarray*}
To show the asymptotic normality of derivative of log-likelihood, we use Cram${\rm \acute{e}}$r-Wold theorem.
For any vector $r= (r_\gamma^\top ,r_\sigma^\top )^\top \in\mathbb{R}^{2(p+K+\xi)}$ with $r_\gamma\in\mathbb{R}^{p+K+\xi}, r_\sigma\in\mathbb{R}^{p+K+\xi}$, we consider 
\begin{eqnarray*}
W=\sqrt{\frac{N}{p_N}}r^\top \left\{\frac{\partial \ell(\theta_0)}{\partial \theta}-E\left[\frac{\partial \ell(\theta_0)}{\partial \theta}\right] \right\}
=
\sum_{i=1}^N W_i,
\end{eqnarray*}
where 
\begin{eqnarray*}
W_i&=&\frac{1}{\sqrt{Np_N}}\left[ r_\gamma^\top \left\{\ell_\gamma(Y_i|X_i,Z_i)A(X_i,Z_i)I(Y_i>0)-E\left[\ell_\gamma(Y|X,Z)A(X,Z)I(Y>0)\right]\right\}\right.\\
&&\left.+r_\sigma^\top \left\{\ell_\sigma(Y_i|X_i,Z_i)A(X_i,Z_i)I(Y_i>0)-E\left[\ell_\sigma(Y|X,Z)A(X,Z)I(Y>0)\right]\right\}\right].
\end{eqnarray*}
From the definition of $W_i$ and Lemma \ref{Hessian}, we see that $V[\sum_{i=1}^N W_i]=r^\top \Sigma r(1+o(1))$.
Our aim is to check whether $\sum_{i=1}^N W_i$ satisfies the Lyapnov's condition of central limit theorem (CLT).
It is easy to find that $E[W_i]=0$. 
Next, from the proof of Lemma \ref{GradientEx}, we have $E_{Y|x,z}[\ell_\gamma(Y|x,z)\mid Y>0]=o(1)$ and  $E_{Y|x,z}[\ell_\sigma(Y|x,z)\mid Y>0]=o(1)$. 
Thus, $V[W_i]=E[W_i^2](1+o(1))$. 
In addition, the property of derivative of log-likelihood, Lemma \ref{splineap} and proof of Lemma \ref{Hessian} yield that
\begin{eqnarray*}
&&E[\{ r_\gamma^\top \{\ell_\gamma(Y_i|X_i,Z_i)A(X_i,Z_i)\}^2]\\
&&= r_\gamma^\top E\left[\frac{2P(Y>0|X,Z)}{(2\gamma_0(X,Z)+1)(\gamma_0(X,Z)+1)}A(X,Z)A(X,Z)^\top  \right]r_\gamma(1+o(1))\\
&=&O(p_N),\\
&&E[\{ r_\sigma^\top \{\ell_\sigma(Y_i|X_i,Z_i)^2X_iB(Z_i)^\top ]\\
&&=
 r_\sigma^\top E\left[\frac{P(Y>0|X,Z)}{2\gamma_0(X,Z)+1}A(X,Z)A(X,Z)^\top \right]r_\sigma(1+o(1))\\
&=& O(p_N)
 \end{eqnarray*}
 and 
 \begin{eqnarray*}
&&E[\{ r_\gamma^\top \{\ell_\gamma(Y_i|X_i,Z_i)A(X_i,Z_i)\}\{ r_\sigma^\top \{\ell_\sigma(Y_i|X_i,Z_i)A(X_i,Z_i)\}]\\
&&=
 r_\gamma^\top E\left[\frac{P(Y>0|X,Z)}{(2\gamma_0(X,Z)+1)(\gamma_0(X,Z)+1)}A(X,Z)A(X,Z)^\top  \right]r_\sigma(1+o(1))\\
&& =O(p_N).
\end{eqnarray*}
Thus, $E[W_i^2]=O(1/N)$.
Lastly, we need to show that 
$$
\frac{1}{\left(\sum_{i=1}^N V[W_i]\right)^{(2+\varepsilon)/2}} \sum_{i=1}^N E[|W_i|^{2+\varepsilon}]\rightarrow 0 
$$
for some $\varepsilon>0$. 
Here, we put $\varepsilon=\delta$, where $\delta$ is given in (S2). 
Note that we can use similar $\varepsilon=\delta$ even for (S1) and (S3). 
To evaluate $E[|W_i|^{2+\delta}]$, we calculate 
\begin{eqnarray*}
&&E_{Y|x,z}[|\ell_\gamma(Y|x,z)|^{2+\delta}\mid Y>0]\\
&&=\int_0^{y^*(x,z)} |\ell_\gamma(y|x,z)|^{2+\delta} \frac{1}{\sigma(x,z)}\left(1+\frac{\gamma_0(x,z) y}{\sigma_{w0}(x,z)}\right)^{-1/\gamma_0(x,z)-1} dy + o(1).
\end{eqnarray*}
The part of $o(1)$ is followed by Conditions (C2), (C4) and the definition of $Q$.
Under (S1), that is, when $\gamma_0(x,z)>0$, it is easy to show that $E_{Y|x,z}[|\ell_\gamma(Y|x,z)|^{2+\delta}\mid Y>0]<\infty$ since $\|\ell_\gamma(Y|x,z)|^{2+\delta}|< C\log (1+y\gamma_0(x,z)/\sigma_{w0}(x,z))$ for some constant $C>0$. 
For (S2), we need to carefully calculate $E_{Y|x,z}[|\ell_\gamma(Y|x,z)|^{2+\delta}\mid Y>0]$.
Since $y^*(x,z)=-\sigma_{w0}(x,z)/\gamma_0(x,z)$, the straightforward caclulation yields that 
$$
E_{Y|x,z}[|\ell_\gamma(Y|x,z)|^{2+\delta}\mid Y>0] <\frac{C}{(2+\delta)\gamma_0(x,z)+1}(1+o(1))
$$
for some constant $C>0$. 
Thus,  from the condition of (S2) : $-1/(2+\delta)<\gamma_0(x,z)<0$, we have $E_{Y|x,z}[|\ell_\gamma(Y|x,z)|^{2+\delta}\mid Y>0]<\infty$. 
For case (S3), $E_{Y|x,z}[|\ell_\gamma(Y|x,z)|^{2+\delta}\mid Y>0]<\infty$ can easily be shown since the distribution of $Y$ is approximated to the exponential distribution.  
Similarly, we can evaluate $E_{Y|x,z}[|\ell_\sigma(Y|x,z)|^{2+\delta}\mid Y>0]<\infty$ under each (S1), (S2) or (S3). 

Next, from the property of normalized $B$-spline function, for any $r_g, r_s\in\mathbb{R}^{K+\xi}$, we have  
$$
E[|r_g^\top  B(Z)+ r_s^\top  B(Z)|^{2+\delta}] \leq  O(K^{(2+\delta)/2} K^{-1}K)=O(K^{(2+\delta)/2}). 
$$
Accordingly, we obtain $ E[|W_i|^{2+\delta}]=O( K^{(2+\delta)/2})$ and 
$$
\sum_{i=1}^N E[|W_i|^{2+\delta}]\leq O\left(K^{(2+\delta)/2} \right).
$$
Consequently, we have
$$
\frac{1}{\left(\sum_{i=1}^N V[W_i]\right)^{(2+\delta)/2}} \sum_{i=1}^N E[|W_i|^{2+\delta}]=O\left(\left(\frac{K}{N}\right)^{(2+\delta)/2}\right)=o(1)
$$
as $N\rightarrow\infty$. 
Thus, Lyapunov's condition of CLT is satisfied. 
This completes the proof. 
\end{proof}

\section{Proof of Theorems}

\begin{proof}[Proof of Theorem \ref{ThmL2}]
Let $(\tilde{X},\tilde{Z})\in{\cal X}\times{\cal Z}$ be random variable with same distribution as $(X_i,Z_i)$ independently. 
According to Appendix A, the difference between $\hat{\gamma}$ and $\gamma_0$ can be shown as
\begin{eqnarray*}
\hat{\gamma}(\tilde{X},\tilde{Z})-\gamma_0(\tilde{X},\tilde{Z})
=\tilde{X}^\top (\hat{\beta}-\beta_0)+B(\tilde{Z})^\top (\hat{b}-b_0) +O(K^{-\zeta}).
\end{eqnarray*}
The condition (C1) implies that $E[|\tilde{X}(\hat{\beta}-\beta_0)^\top \||^2]\leq C E[\|\hat{\beta}-\beta_0\|^2]$ for some constant $C>0$.
Lemma \ref{Bsplinematrix} and Cauchy-Schwarz inequality yield that 
$$
E\|B(\tilde{Z})(\hat{b}-b_{0})\|^2] \leq  M_{\max} E[\|\hat{b}-b_0\|^2] .
$$
Thus, the asymptotic rate of $\|\hat{\gamma}-\gamma_0\|_{L_2}$ is dominated by that of $\|\hat{\theta}-\theta_0\|$. 
From Lemma \ref{RatePara}, we obtain 
$$
\|\hat{\gamma}-\gamma_0\|_{L_2}\leq  O\left(\sqrt{\frac{K}{Np_N}}\right)+O(K^{-m})+O\left(p_N^{-\rho}\right)+O(K^{-\zeta}).
$$
Here, we note that the term $O(K^{-\zeta})$ is negrible order since $\zeta>m$.
Similarly, we have
$$
\left\|\log \hat{\sigma}-\log\sigma_{w0}\right\|_{L_2}\leq   O\left(\sqrt{\frac{K}{Np_N}}\right)+O(K^{-m})+O\left(p_N^{-\rho}\right).
$$
Thus, the first assertion of Theorem \ref{ThmL2} is shown.
The remaining two assetions are obtained from straightforward calculation.
\end{proof}

%%%%%%%

\begin{proof}[Proof of Theorem \ref{Linfty}]

We write
$$
\Sigma^{-1}
=\left[
\begin{array}{cc}
S_{11}&S_{12}\\
S_{21}& S_{22}
\end{array}\right],
$$
where each $S_{ij}(i,j=1,2)$ is $(p+K+\xi)$-square matrix.  
From Lemma \ref{Hessian}, we have $S_{ij}=O(p_N^{-1})$ for $i,j=1,2$. 
Similar to the proof of Lemma \ref{RatePara}, we have that for $(x,z)\in{\cal X}\times{\cal Z}$,
\begin{eqnarray*}
\hat{\gamma}(x,z)-\gamma_0(x,z)
&=&
 A(x,z)^\top 
\left[
\begin{array}{c}
\hat{\beta}-\beta_0\\
\hat{b}-b_0
\end{array}\right]\\
&=& 
 A(x,z)^\top \left\{S_{11} p_N^{-1}
  \left[
\begin{array}{c}
\frac{\partial \ell_{pen}(\theta_0)}{\partial \beta}\\
\frac{\partial \ell_{pen}(\theta_0)}{\partial b}
\end{array}\right]
+
S_{12}p_N^{-1}
  \left[
\begin{array}{c}
\frac{\partial \ell_{pen}(\theta_0)}{\partial u}\\
\frac{\partial \ell_{pen}(\theta_0)}{\partial c}
\end{array}\right]
\right\}+O(K^{-\zeta}).
\end{eqnarray*}
From the proof of Lemma \ref{AS.para}, we obtain 
\begin{eqnarray*}
&&\hat{\gamma}(x,z)-\gamma(x,z)\\
&&=
\frac{1}{Np_N}\sum_{i=1}^NI(Y_i>0)\\
&&\hspace{5mm}\times\left\{ \ell_\gamma(Y_i|X_i,Z_i)A(x,z)^\top  S_{11}A(X_i,Z_i)+\ell_\sigma(Y_i|X_i,Z_i)A(x,z)^\top  S_{12}A(X_i,Z_i)\right\}\\
&& +A(x,z)^\top S_{11}
\left[
\begin{array}{c}
\underline{0}_p\\
p_N^{-1}K^{2m}\Omega(\lambda) b_0
\end{array}\right]
+
A(x,z)^\top S_{12}
\left[
\begin{array}{c}
\underline{0}_p\\
p_N^{-1}K^{2m}\Omega(\nu) c_0
\end{array}\right]+O(K^{-\zeta}),
\end{eqnarray*}
where $\underline{0}_p$ is $p$-zero vector, $\ell_\gamma(y|x,z)=\ell_\gamma(y|\bar{\gamma}_0(x,z),\bar{\sigma}_{w0}(x,z))$, $\ell_\sigma(y|x,z)=\ell_\sigma(y|\bar{\gamma}_0(x,z),\bar{\sigma}_{w0}(x,z))$ and $\Omega(\cdot)$ is that appeared in Lemma \ref{GradientEx}. 
From definition of normalized $B$-spline basis, we have $\sup_{(x,z)\in{\cal X}\times{\cal Z}}\|A(x,z)\|=O(\sqrt{K})$. 
Therefore, (C5) and Lemma \ref{GradientEx} imply that 
\begin{eqnarray*}
&&A(x,z)^\top S_{11}
\left[
\begin{array}{c}
\underline{0}_p\\
p_N^{-1}K^{2m}\Omega(\lambda) b_0
\end{array}\right]
+
A(x,z)^\top S_{12}
\left[
\begin{array}{c}
\underline{0}_p\\
p_N^{-1}K^{2m}\Omega(\lambda) c_0
\end{array}\right]\\
&&=O(p_N^{-1}\lambda K^m)+O(p_N^{-1}\nu K^m)\\
&&=O(K^{-m}). 
\end{eqnarray*}
Define 
\begin{eqnarray*}
G_{i}(x,z)&=&p_N^{-1}\{\ell_\gamma(Y_i|X_i,Z_i)A(x,z)^\top  S_{11}A(X_i,Z_i)\\
&&+\ell_\sigma(Y_i|X_i,Z_i)A(x,z)^\top  S_{12}A(X_i,Z_i)\}I(Y_i>0).
\end{eqnarray*}
Lemma \ref{GradientEx} implies that $\sup_{(x,z)\in{\cal X}\times{\cal Z}} |E[G_{i}(x,z)]|<O(p_N^{-\rho})$. 

The remain of proof is to show 
\begin{eqnarray}
\sup_{(x,z)\in{\cal X}\times{\cal Z}}
\left|\frac{1}{N}\sum_{i=1}^N G_i(x,z)-E[G_{i}(x,z)]\right|=O(\sqrt{K\log N/Np_N}). \label{StLinf}
\end{eqnarray}
Let $E_i=\gamma_0(X_i,Z_i)^{-1}\log (1+Y_i\gamma_0(X_i,Z_i)/\sigma_{w0}(X_i,Z_i))$ for the cases (S1) and (S2). 
For the case (S3): $\gamma_0(x,z)=0$, we set $E_i=Y_i/\sigma_{w0}(X_i,Z_i)$. 
Then, under $Y_i>0$, $E_i$ is asymptotically distributed as the standard exponential distribution.
Therefore, for constant $M>0$, we have $P(E_i>M)=e^{-M}(1+o(1))$. 
Define the event ${\cal E}=\{ \{\max_i E_i<M\} \cap \{Y_i>0,i=1,\ldots,N\}\}$ and the sequence $\varepsilon_N=C_\varepsilon\sqrt{K\log N/Np_N}$, where $C_\varepsilon>0$ is the constant defined below.
Then, we have $P({\cal E}^c)=1-(1-e^{-M})^{Np_N}(1+o(1))$. 
Therefore, $P({\cal E}^c)\rightarrow 0$ as long as $M>\log(Np_N\log(N))$.
We then have 
\begin{eqnarray*}
&&\sup_{(x,z)\in{\cal X}\times{\cal Z}}
\left|\frac{1}{N}\sum_{i=1}^N G_i(x,z)-E[G_i(x,z)]\right|\\
&&\leq 
\sup_{(x,z)\in{\cal X}\times{\cal Z}}
\left|\frac{1}{N}\sum_{i=1}^N G_i(x,z)I({\cal E})-E[G_i(x,z)I({\cal E})]\right|\\
&&\quad+
\sup_{(x,z)\in{\cal X}\times{\cal Z}}
\left|\frac{1}{N}\sum_{i=1}^N G_i(x,z)I({\cal E}^c)-E[G_i(x,z)I({\cal E}^c)]\right|.
\end{eqnarray*}
Since  $\sup_{(x,z)\in{\cal X}\times{\cal Z}}\|A(x,z)\|=O(\sqrt{K})$, we have $\sup_{(x,z)\in{\cal X}\times{\cal Z}}V[ G_i(x,z)]\leq O(K/Np_N)=o(\varepsilon_N^2)$. 
Therefore, we obtain
$$
\sup_{(x,z)\in{\cal X}\times{\cal Z}}
\left|\frac{1}{N}\sum_{i=1}^N G_i(x,z)I({\cal E}^c)-E[G_i(x,z)I({\cal E}^c)]\right|
\leq o(\varepsilon_N\{1-(1-e^{-M})\}^{Np_N})=o(\varepsilon_N).
$$
Thus, all that remains is to show 
$$
\sup_{(x,z)\in{\cal X}\times{\cal Z}}
\left|\frac{1}{N}\sum_{i=1}^N G_i(x,z)I({\cal E})-E[G_i(x,z)I({\cal E})]\right|=O(\varepsilon_N).
$$
We take $J$-fixed points $(x^*_j,z^*_j)\in{\cal X}\times{\cal Z} (j=1,\ldots,J)$ and define the set ${\cal X}_j\times{\cal Z}_j=\{(x,z)| \|(x,z)-(x_j^*,z_j^*)\|<(Np_N)^{-\eta}\}$with $\eta>0$ for $j=1,\ldots,J$. 
The number $J$ is quite large in order to satisfy ${\cal X}\times{\cal Z}\subset \bigcup_{j=1}^J {\cal X}_j\times{\cal Z}_j$.
From Lemma 2.5 of van der Geer (2000), at least, for some constant $C_\eta>0$, $J\leq  C_\eta (Np_N)^{\eta(p+d)}$ holds.
Then, 
\begin{eqnarray*}
&&\sup_{(x,z)\in{\cal X}\times{\cal Z}}
\left|\frac{1}{n}\sum_{i=1}^n \{G_i(x,z)I({\cal E})-E[G_i(x,z)I({\cal E})\}]\right|\\
&&= \max_j \sup_{(x,z)\in {\cal X}_j\times{\cal Z}_j }\Bigl|\frac{1}{n}\sum_{i=1}^n \{G_i(x,z)I({\cal E})-E[G_i(x,z)I({\cal E})]\}\\
&&\hspace{5cm}-\{G_i(x_j^*,z_j^*)I({\cal E})-E[G_i(x_j^*,z_j^*)I({\cal E})]\}\Bigr|\\
&&\quad + \max_j \left|\frac{1}{n}\sum_{i=1}^n G_i(x_j^*,z_j^*)I({\cal E})-E[G_i(x_j^*,z_j^*)I({\cal E})]\right|.
\end{eqnarray*}
Here, we choose $\eta$ such that $(Np_N)^{-\eta} = O(K^{-2})$. 
Then, from the  Lipschitz continuity of $B$-spline basis (see, de Boor 2001) yields that $B(z)-B(z_j^*)$ has an order $O(K^{3/2}(Np_N)^{-\eta})=O(K^{-1/2})$ for only few element of and zero for other many elements. 
Therefore, on the event ${\cal E}$, we have 
\begin{eqnarray*}
&& \max_j \sup_{(x,z)\in {\cal X}_j\times{\cal Z}_j }\left|\frac{1}{N}\sum_{i=1}^N \{G_i(x,z)I({\cal E})-E[G_i(x,z)I({\cal E})]\}\right.\\
&&\left.\quad\quad\quad\quad\quad\quad\quad\quad\quad -\{G_i(x_j^*,z_j^*)I({\cal E})-E[G_i(x_j^*,z_j^*)I({\cal E})]\}\right|\\
&&\leq O_P(\sqrt{K (Np_N)^{-1}})=o_P(\varepsilon_N)
\end{eqnarray*}
from the proof of Lemma \ref{GradientSt}.
Lastly, we show that 
\begin{eqnarray}
&&P\left(\max_j \left|\frac{1}{N}\sum_{i=1}^N G_i(x_j^*,z_j^*)I({\cal E})-E[G_i(x_j^*,z_j^*)I({\cal E})]\right|>\varepsilon_N\right)\nonumber\\
&&\leq\sum_{j=1}^{J} P\left(\left|\frac{1}{N}\sum_{i=1}^N G_i(x_j^*,z_j^*)I({\cal E})-E[G_i(x_j^*,z_j^*)I({\cal E})]\right|>\varepsilon_N\right)\nonumber\\
&&\rightarrow 0. \label{CaseShow}
\end{eqnarray}
by using Lemma \ref{bernstein}. 

It can easily be described that for any $(x,z)\in{\cal X}\times{\cal Z}$, $V[N^{-1}G_i(x,z)]< C_1p_N^{-1} N^{-2}K$ for some constant $C_1>0$ for all (S1)--(S3). 
We now consider the case (S1). 
Under ${\cal E}$, we have $|\ell_\gamma(Y_i|Z_i,X_i)|\leq M/\gamma_{min}$ and $|\ell_\sigma(Y_i|X_i,Z_i)|<1$. 
Together with $A(x,z)=O(\sqrt{K})$, we see that $|N^{-1}\{G_i(x,z)I({\cal E})-E[G_i(x,z)I({\cal E})]\}|\leq C_2 K p_N^{-1}N^{-1}M$. 
Putting $M=1/\varepsilon_N=O(\sqrt{Np_N/K\log N})$, we have $P({\cal E}^c)=o(1)$. 
In addition, from Lemma \ref{bernstein}, we obtain 
\begin{eqnarray*}
&& P\left(\left|\frac{1}{N}\sum_{i=1}^N G_i(x_j^*,z_j^*)I({\cal E})-E[G_i(x_j^*,z_j^*)I({\cal E})]\right|>\varepsilon_N\right)\\
&& \leq 2\exp\left[\frac{2^{-1}\varepsilon_N^2}{C_1 K/(Np_N)+ 3^{-1}C_2\varepsilon_N M K^{1/2}/(Np_N)}\right]\\
&&\leq 2 \exp\left[-C^*C_\varepsilon\log N\right]
\end{eqnarray*} 
for some constant $C^*>0$. 
Since $J=O((Np_N)^{\eta(p+d)})$, if we choose $C_{\varepsilon}$ such that $(Np_N)^{2(p+d)}/N^{C^*C_\varepsilon}\rightarrow 0$,  (\ref{CaseShow}) holds. 

We next focus on the case (S2): $-(2+\delta)^{-1}<\gamma_0(x,z)<0$. 
Under ${\cal E}$, we obtain $|\ell_\gamma(Y_i|X_i,Z_i)|<C e^{M/(2+\delta)}$ and $|\ell_\sigma(Y_i|X_i,Z_i)|<C e^{M/(2+\delta)}$ for some constants $C>0$. 
Thus, for any fixed point $(x,z)\in{\cal X}\times {\cal Z}$, 
\begin{eqnarray*}
\left|\frac{1}{N}\sum_{i=1}^N G_i(x,z)I({\cal E})-E[G_i(x,z)I({\cal E})]\right|
&\leq& C_2e^{M/(2+\delta)} K(Np_N)^{-1}
\end{eqnarray*}
for some constant $C_2>0$. 
Therefore, we put $M=\log(Np_N\log N)$. 
This implies $P({\cal E}^c)\rightarrow 0$ and under the condition that $\{K\log N\}^{1+\delta/2}/(Np_N)^{\delta/2}\rightarrow 0$, $e^{M/(2+\delta)}=(Np_N)^{1/(2+\delta)}\{\log N\}^{1/(2+\delta)}\leq O(\varepsilon_N^{-1})$.
Consequently, Lemma \ref{bernstein} shows that
\begin{eqnarray*}
&& P\left(\left|\frac{1}{N}\sum_{i=1}^N G_i(x_j^*,z_j^*)I({\cal E})-E[G_i(x_j^*,z_j^*)I({\cal E})]\right|>\varepsilon_N\right)\\
&& \leq 2\exp\left[\frac{2^{-1}\varepsilon_N^2}{C_1 K/(Np_N)+ 3^{-1}C_2\varepsilon_N e^{M/(2+\delta)} K/(Np_N)}\right]\\
&&\leq 2 \exp\left[-C^*C_\varepsilon\log N\right].
\end{eqnarray*} 
Similar to the case (S1), we obtain (\ref{CaseShow}). 

For the case (S3): $\gamma_0(x,z)=0$, we obtain $|\ell_\gamma(Y_i|X_i,Z_i)|<C M^2$ and $|\ell_\sigma(Y_i|X_i,Z_i)|<C M$ for some constants $C>0$.
When putting $M=\sqrt{1/\varepsilon_N}$, $P({\cal E}^c)\rightarrow 0$  and (\ref{CaseShow}) can be shown as the same mannar as case (S1). 
Thus, in each (S1), (S2) or (S3), (\ref{StLinf}) was proven. 
Consequently, we obtain
$$
\|\hat{\gamma}-\gamma_0\|_{L_\infty} \leq O\left(\sqrt{\frac{K\log N}{Np_N}}\right)+O(K^{-m}).+O(p_N^{-\rho}).
$$
Similarly, the rate of convergence of $\|\log\hat{\sigma}-\log\sigma_0\|_{L_\infty}$ can be derived. 
Thus, the proof is completed.
\end{proof}

\begin{proof}[Proof of Theorem \ref{Norm}]
From Lemma \ref{splineap}, we have
\begin{eqnarray*}
\left[
\begin{array}{c}
\hat{\gamma}(x,z)-\gamma_0(x,z)\\
\log \hat{\sigma}(x,z)-\log \sigma_{w0}(x,z)
\end{array}
\right]
=
\left[
\begin{array}{cc}
A(x,z)^\top & 0_{p+K+\xi}^\top \\
0_{p+K+\xi}^\top  &A(x,z)^\top 
\end{array}
\right]
(\hat{\theta}-\theta_0)+O(K^{-\zeta}).
\end{eqnarray*}
Under the condition of Theorem \ref{Norm}, $K^{-\zeta}=o(\sqrt{N/Np_N})$.
Similar to proof of Lemma \ref{RatePara}, Taylor expansion yields that 
\begin{eqnarray*}
\hat{\theta}-\theta_0&=&\Sigma^{-1} \frac{1}{p_N}\frac{\partial \ell_{pen}(\theta_0)}{\partial \theta} (1+o_P(1))\\
&=&
\Sigma^{-1} \frac{1}{p_N}\left[\left\{\frac{\partial \ell(\theta_0)}{\partial \theta}-E\left[\frac{\partial \ell(\theta_0)}{\partial \theta}\right]\right\}+E\left[\frac{\partial \ell(\theta_0)}{\partial \theta}\right]+\Omega_{\gamma,\sigma}\theta_0\right](1+o_P(1)).
\end{eqnarray*}
Here, we note that $A(x,z)=O(\sqrt{K})$. 
From the proof of Lemma \ref{GradientEx} and condition that 
\[(Np_N/K)^{1/2}\{p_N^{-\rho}, K^{-m}\}\rightarrow 0,
\]  we obtain 
\begin{eqnarray*}
&&\sqrt{\frac{Np_N}{K}}\left[
\begin{array}{cc}
A(x,z)^\top & 0_{p+K+\xi}^\top \\
0_{p+K+\xi}^\top  &A(x,z)^\top 
\end{array}
\right]\Sigma^{-1}\left\{p_N^{-1}E\left[\frac{\partial \ell(\theta_0)}{\partial \theta}\right]+p_N^{-1}\Omega_{\gamma,\sigma}\theta_0\right\}\\
&&\leq
O\left(\sqrt{\frac{Np_N}{K}}K^{1/2}\left\{K^{-1/2}p_N^{-\rho}+ K^{-m-1/2}\right\}\right)\\
&&=o(1).
\end{eqnarray*}
Next, Lemma \ref{AS.para} yields that 
\begin{eqnarray*}
&&\sqrt{\frac{Np_N}{K}}
\left[
\begin{array}{cc}
A(x,z)^\top & 0_{p+K+\xi}^\top \\
0_{p+K+\xi}^\top  &A(x,z)^\top 
\end{array}
\right]
\Sigma^{-1}
\frac{1}{p_N}\left\{\frac{\partial \ell(\theta_0)}{\partial \theta}-E\left[\frac{\partial \ell(\theta_0)}{\partial \theta}\right]\right\}\\
&& \stackrel{{\cal D}}{\rightarrow} N(0, \lim_{N\rightarrow\infty}D(x,z)^\top \Sigma^{-1}D(x,z)/K).
\end{eqnarray*}
Consequently, we obtain 
$$
\sqrt{\frac{Np_N}{K}} \left[
\begin{array}{cc}
\hat{\gamma}(x,z)-\gamma_0(x,z)\\
\log \hat{\sigma}(x,z)-\log\sigma_{w0}(x,z)
\end{array}
\right]
 \stackrel{{\cal D}}{\rightarrow} N(0, \lim_{N\rightarrow\infty}D(x,z)^\top \Sigma^{-1}D(x,z)/K).
$$
Define $g(a)=e^a$ for any $a\in\mathbb{R}$. 
We then apply the delta method to $g(\log \hat{\sigma}(x,z))-g(\log\sigma_{w0}(x,z))$, we obtain 
$$
\sqrt{\frac{Np_N}{K}} \left[
\begin{array}{cc}
\hat{\gamma}(x,z)-\gamma_0(x,z)\\
\displaystyle\frac{\hat{\sigma}(x,z)}{\sigma_{w0}(x,z)}-1
\end{array}
\right]
 \stackrel{{\cal D}}{\rightarrow} N(0, \lim_{N\rightarrow\infty}D(x,z)^\top \Sigma^{-1}D(x,z)/K),
$$
which completes the proof.
\end{proof}

\begin{proof}[Proof of Theorem \ref{ParametricPart}]
Let $\theta_{P}=(\beta^\top ,u^\top )^\top \in\mathbb{R}^{2p}$. 
By the Taylor expansion to $\partial \ell_{pen}(\beta,\hat{b},u,\hat{c})/\partial\theta_{P}$ around $(\hat{\beta}^\top ,\hat{u}^\top )^\top =(\beta_0^\top ,u_{w,0}^\top )^\top $,
we have  
\begin{eqnarray*}
0&=&\frac{\partial \ell_{pen}(\hat{\beta},\hat{b},\hat{u},\hat{c})}{\partial \theta_{P}}\\
&=&
\frac{1}{p_N}\frac{\partial \ell_{pen}(\beta_0,\hat{b},u_{w,0},\hat{c})}{\partial  \theta_{P}}
+\frac{1}{p_N}\left(\frac{\partial^2 \ell_{pen}(\beta_0,\hat{b},u_{w,0},\hat{c})}{\partial \theta_P\partial\theta_{P}^\top }\right)
\left[
\begin{array}{c}
\hat{\beta}-\beta_0\\
\hat{u}-u_{w,0}
\end{array}
\right](1+o_P(1)).
\end{eqnarray*}
From Lemma \ref{RatePara}, we have $\|\hat{b}-b_0\|+\|\hat{c}-c_0\|\stackrel{P}{\rightarrow}0$. 
Thefore, Lemma \ref{Hessian} implies that 
$$
\frac{1}{p_N}\frac{\partial^2 \ell_{pen}(\beta_0,\hat{b},u_{w,0},\hat{c})}{\partial \theta_P\partial\theta_{P}^\top }
=\Sigma_{\beta,u}(1+o_P(1)). 
$$
We now denote
\begin{eqnarray*}
\frac{1}{p_N}\frac{\partial \ell_{pen}(\beta_0,\hat{b},u_{w,0},\hat{c})}{\partial  \theta_{P}}
=
\frac{1}{p_N}\frac{\partial \ell_{pen}(\beta_0,b_0,u_{w,0},c_0)}{\partial  \theta_{P}}+R_N(\hat{b},\hat{c}),
\end{eqnarray*}
where 
\begin{eqnarray*}
R_N(\hat{b},\hat{c})
=\frac{1}{p_N}\left(\frac{\partial \ell_{pen}(\beta_0,\hat{b},u_{w,0},\hat{c})}{\partial  \theta_{P}}-\frac{\partial \ell_{pen}(\beta_0,b_0,u_{w,0},c_0)}{\partial  \theta_{P}}\right).
\end{eqnarray*}
Then, we have 
\begin{eqnarray*}
\left[
\begin{array}{c}
\hat{\beta}-\beta_0\\
\hat{u}-u_{w,0}
\end{array}
\right]
=-\Sigma_{\beta,u}^{-1}\left(\frac{1}{p_N}\frac{\partial \ell_{pen}(\beta_0,b_0,u_{w,0},c_0)}{\partial  \theta_{P}}\right)-\Sigma_{\beta,u}^{-1}R_N(\hat{b},\hat{c}).
\end{eqnarray*}
Similar to the proof of Lemmas \ref{GradientEx}, \ref{AS.para} and Theorem \ref{Norm}, 
we obtain $p_N^{-1}E[\partial \ell_{pen}(\beta_0,b_0,u_{w,0},c_0)/\partial  \theta_{P}]= O(p_N^{-\rho})$ and 
\begin{eqnarray*}
\sqrt{Np_N}\Sigma_{\beta,u}^{-1}\left(\frac{1}{p_N}\frac{\partial \ell_{pen}(\beta_0,b_0,u_{w,0},c_0)}{\partial  \theta_{P}}\right) \stackrel{D}{\rightarrow} N(0,\Sigma_{\beta,u}).
\end{eqnarray*}
Thus, if $\sqrt{Np_N}R_N(\hat{b}, \hat{c})=o_P(1)$ and $E[R_N(\hat{b},\hat{c})]=o(p_N^{-\rho})$, the theorem is said to be proven.
Since $\|\hat{b}-b_0\|+\|\hat{c}-c_0\|\stackrel{P}{\rightarrow}0$ from Lemma \ref{ConvPara} and $R_N(b,c)$ is continuous with respect to $(b,c)$, the standard deviation of $R_N(\hat{b},\hat{c})$ becomes $o(1/\sqrt{Np_N})$. 
The remaining part of the proof is only to show $E[R_N(\hat{b},\hat{c})]=o(p_N^{-\rho})$.

In following, we only consider the case (S1) and (S2). 
The proof for (S3) is similar and it is omited for the sake of space.
Let $\gamma(x,z|b)=x^\top \beta_0+ B(z)^\top b$ for any $b\in\mathbb{R}^{d(K+\xi)}$ and $\sigma(x,z|c)=\exp[x^\top u_{w,0}+B(z)^\top c]$ for any $c\in\mathbb{R}^{d(K+\xi)}$. 
Note that $\gamma_0(x,z)=\gamma(x,z|b_0)$ and $\sigma_{w0}(x,z)=\sigma(x,z|c_0)$.
We further let 
$$
\ell_\gamma(y|x,z,\hat{b},\hat{c})
=(\gamma(x,z|\hat{b})^{-1}+1)\frac{y/\sigma(x,z|\hat{c})}{1+y\gamma(x,z|\hat{b})/\sigma(x,z|\hat{c})}-\gamma(x,z|\hat{b})^{-2}\log\left(1+\frac{y\gamma(x,z|\hat{b})}{\sigma(x,z|\hat{c})}\right)
$$
and 
$$
\ell_\sigma(y|x,z,\hat{b},\hat{c})
=1-(\gamma(x,z|\hat{b})^{-1}+1)\frac{y \gamma(x,z|\hat{b})/\sigma(x,z|\hat{c})}{1+y \gamma(x,z|\hat{b})/\sigma(x,z|\hat{c}) }.
$$
Then, we have 
$$
R_N(\hat{b},\hat{c})
=
\frac{1}{Np_N}\sum_{i=1}^N \left[
\begin{array}{c}
\{\ell_\gamma(Y_i|X_i,Z_i,\hat{b},\hat{c})-\ell_\gamma(Y_i|X_i,Z_i,b_0,c_0)\}X_i\\
\{\ell_\sigma(Y_i|X_i,Z_i,\hat{b},\hat{c})-\ell_\sigma(Y_i|X_i,Z_i,b_0,c_0)\}X_i\\
\end{array}
\right]I(Y_i>0).
$$
By the definition of $R_N$, we need to prove
\begin{eqnarray}
p_N^{-1} E[\{\ell_\gamma(Y_i|X_i,Z_i,\hat{b},\hat{c})-\ell_\gamma(Y_i|X_i,Z_i,b_0,c_0)\}X_iI(Y_i>0)]=o(\rho_N^{-\rho}) \label{LastOrder1}
\end{eqnarray}
and 
\begin{eqnarray}
p_N^{-1} E[\{\ell_\sigma(Y_i|X_i,Z_i,\hat{b},\hat{c})-\ell_\sigma(Y_i|X_i,Z_i,b_0,c_0)\}X_iI(Y_i>0)]=o(\rho_N^{-\rho}). \label{LastOrder2}
\end{eqnarray}
We now focus on deriving (\ref{LastOrder1}) since the proof of (\ref{LastOrder2}) is similar.
The Taylor expansion implies that 
\begin{eqnarray*}
\ell_\gamma(Y_i|X_i,Z_i,\hat{b},\hat{c})-\ell_\gamma(Y_i|X_i,Z_i,b_0,c_0)
&=& \frac{\partial \ell_\gamma(Y_i|X_i,Z_i,b_0,c_0) }{\partial b}(\hat{b}-b_0)(1+o_P(1))\\
&&+ \frac{\partial \ell_\gamma(Y_i|X_i,Z_i,b_0,c_0) }{\partial c}(\hat{c}-c_0)(1+o_P(1)).
\end{eqnarray*}
Here, we obtain 
\[
\frac{\partial \ell_\gamma(Y_i|X_i,Z_i,b_0,c_0) }{\partial b}
=\ell_{\gamma\gamma}(Y_i|X_i,Z_i,b_0,c_0) B(Z_i),
\]
where 
\begin{eqnarray*}
\ell_{\gamma\gamma}(y|x,z,b,c) 
&=&
-2\gamma(x,z|b)^{-2}\frac{y/\sigma(x,z|c)}{1+y\gamma(x,z|b)/\sigma(x,z|c)}\\
&&-(\gamma(x,z|b)^{-1}+1)\frac{-y^2/\sigma^2(x,z|c)}{\{1+y\gamma(x,z|b)/\sigma(x,z|c)\}^2}\\
&&\quad +2\gamma(x,z|b)^{-3}\log\left(1+\frac{y\gamma(x,z|b)}{\sigma(x,z|c)}\right).
\end{eqnarray*}
We note that $E[\ell_{\gamma\gamma}(Y_i|x,z,b_0,c_0)\mid Y_i>0]=2(2\gamma_0(x,z)+1)^{-1}(\gamma_0(x,z)+1)^{-1}(1+o(1))$, which is related to Fisher information matrix of $-\log h(Y_i|\gamma_0(x,z),\sigma_0(x,z))$ (see, the proof of Lemma \ref{GradientSt}).
Similarly, we have 
\[
\frac{\partial \ell_\gamma(Y_i|X_i,Z_i,b_0,c_0) }{\partial c}
=\ell_{\gamma\sigma}(Y_i|X_i,Z_i,b_0,c_0) B(Z_i),
\]
where 
\begin{eqnarray*}
&&\ell_{\gamma\sigma}(Y_i|X_i,Z_i,b,c)\\
&&=\gamma(x,z|b)^{-2}\frac{y\gamma(x,z|b)/\sigma(x,z|c)}{1+y\gamma(x,z|b)/\sigma(x,z|c)}\\
&&
\quad -(\gamma(x,z|b)^{-1}+1)\left\{\frac{y/\sigma_0(x,z|c)}{1+y\gamma(x,z|b)/\sigma(x,z|c)^2}-\frac{y^2\gamma(x,z|b)/\sigma(x,z|c)^2}{\{1+y\gamma(x,z|b)/\sigma(x,z|c)^2\}^2}\right\}
\end{eqnarray*}
and $E[\ell_{\gamma\sigma}(Y_i|x,z,b_0,c_0)\mid Y_i>0]=(2\gamma_0(x,z)+1)^{-1}(1+o(1))$.
We write 
$$
\Sigma^{-1}
=
\left[
\begin{array}{cccc}
\bar{\Sigma}_{\beta\beta}&\bar{\Sigma}_{\beta b}&\bar{\Sigma}_{\beta u}&\bar{\Sigma}_{\beta c}\\
\bar{\Sigma}_{b\beta}&\bar{\Sigma}_{b b}&\bar{\Sigma}_{b u}&\bar{\Sigma}_{b c}\\
\bar{\Sigma}_{u\beta}&\bar{\Sigma}_{u b}&\bar{\Sigma}_{u u}&\bar{\Sigma}_{u c}\\
\bar{\Sigma}_{c\beta}&\bar{\Sigma}_{c b}&\bar{\Sigma}_{c u}&\bar{\Sigma}_{c c}
\end{array}
\right],
$$
where the size of each block of $\Sigma^{-1}$ are similar to the lengrh of vector appeared in the indeces $\beta,u\in\mathbb{R}^p, b,c\in\mathbb{R}^{d(K+\xi)}$. 
Since 
\begin{eqnarray*}
\left[
\begin{array}{cccc}
\hat{\beta}-\beta_0\\
\hat{b}-b_0\\
\hat{u}-u_0\\
\hat{c}-c_0
\end{array}
\right]
=
\Sigma^{-1}
\left[
\begin{array}{c}
\frac{\partial}{\partial\beta}\ell_{pen}(\theta_0)\\
\frac{\partial}{\partial b}\ell_{pen}(\theta_0)\\
\frac{\partial}{\partial u}\ell_{pen}(\theta_0)\\
\frac{\partial}{\partial c}\ell_{pen}(\theta_0)
\end{array}
\right](1+o_P(1)),
\end{eqnarray*}
we obtain 
\begin{eqnarray*}
&&\hat{b}-b_0\\
&&=\left[ \bar{\Sigma}_{b\beta}\frac{\partial}{\partial\beta}\ell_{pen}(\theta_0)
+\bar{\Sigma}_{b b}\frac{\partial}{\partial b}\ell_{pen}(\theta_0)+ \bar{\Sigma}_{b u}\frac{\partial}{\partial u}\ell_{pen}(\theta_0)
+\bar{\Sigma}_{b c}\frac{\partial}{\partial c}\ell_{pen}(\theta_0)\right](1+o_P(1))\\
&&=
\frac{1}{N}\sum_{i=1}^N \left[\ell_\gamma(Y_i|X_i,Z_i)\{\bar{\Sigma}_{b\beta}X_i+ \bar{\Sigma}_{b b}B(Z_i) \}+\ell_\sigma(Y_i|X_i,Z_i)\{\bar{\Sigma}_{bu}X_i+ \bar{\Sigma}_{b c}B(Z_i) \}\right.\\
&&\left. \quad\quad\quad+K^{2m}\bar{\Sigma}_{b b}\Omega(\lambda) b_0+ K^{2m}\bar{\Sigma}_{bc}\Omega(\nu) c_0 \right]I(Y_i>0)(1+o_P(1))\\
&&\equiv\frac{1}{N}\sum_{i=1}^N L_b(Y_i|X_i,Z_i)I(Y_i>0) (1+o_P(1))
\end{eqnarray*}
and 
\begin{eqnarray*}
&&\hat{c}-c_0\\
&&=
\frac{1}{N}\sum_{i=1}^N \left[\ell_\gamma(Y_i|X_i,Z_i)\{\bar{\Sigma}_{c\beta}X_i+ \bar{\Sigma}_{c b}B(Z_i) \}+\ell_\sigma(Y_i|X_i,Z_i)\{\bar{\Sigma}_{cu}X_i+ \bar{\Sigma}_{c c}B(Z_i) \}\right.\\
&&\left. \quad\quad\quad+ K^{2m}\bar{\Sigma}_{c b}\Omega(\lambda) b_0+ K^{2m}\bar{\Sigma}_{cc}\Omega(\nu) c_0 \right]I(Y_i>0)(1+o_P(1))\\
&&\equiv\frac{1}{N}\sum_{i=1}^N L_c(Y_i|X_i,Z_i)I(Y_i>0) (1+o_P(1))
\end{eqnarray*}
where $\ell_\gamma(Y_i|X_i,Z_i)=\ell_\gamma(Y_i|X_i,Z_i,b_0,c_0)$ and $\ell_\sigma(Y_i|X_i,Z_i)=\ell_\sigma(Y_i|X_i,Z_i,b_0,c_0)$. 
Therefore, (\ref{LastOrder1}) can be written by
\begin{eqnarray*}
&&p_N^{-1}E\left[\{\ell_\gamma(Y_i|X_i,Z_i,\hat{b},\hat{c})-\ell_\gamma(Y_i|X_i,Z_i,b_0,c_0)\}X_i I(Y_i>0)\right]\\
&&=\frac{1}{Np_N}\sum_{j=1}^N E[\ell_{\gamma\gamma}(Y_i|X_i,Z_i,b_0,c_0)B(Z_i)^\top  L_b(Y_j|X_j,Z_j) X_iI(Y_i>0)]\\
&&+\frac{1}{Np_N}\sum_{j=1}^N E[\ell_{\gamma\sigma}(Y_i|X_i,Z_i,b_0,c_0)B(Z_i)^\top  L_c(Y_j|X_j,Z_j) X_i I(Y_i>0)]
\end{eqnarray*}
By the similar arguments as Lemma \ref{GradientEx}, we have $E[L_b(Y_j|X_j,Z_j)\mid Y_i>0]=O(K^{-m})+O(p_N^{-\rho})$ and $E[L_c(Y_j|X_j,Z_j)\mid Y_i>0]=O(K^{-m})+O(p_N^{-\rho})$. 
Furthermore, from Cauchy–Schwarz inequality, we obtain $E[\ell_{\gamma\gamma}(Y_i|X_i,Z_i,b_0,c_0)L_b(Y_j|X_j,Z_j)\mid Y_i>0]\leq O(K^{-m})+O(p_N^{-\rho})$ and $E[\ell_{\gamma\sigma}(Y_i|X_i,Z_i,b_0,c_0)L_c(Y_j|X_j,Z_j)\mid Y_i>0]\leq O(K^{-m})+O(p_N^{-\rho})$. 
This implies that 

\begin{eqnarray*}
&&\frac{1}{Np_N}\sum_{j=1}^N E[\ell_{\gamma\gamma}(Y_i|X_i,Z_i,b_0,c_0)B(Z_i)^\top  L_b(Y_j|X_j,Z_j) X_iI(Y_i>0)]\\
&&=
\frac{1}{Np_N} E[P(Y_i\mid X,Z)\ell_{\gamma\gamma}(Y_i|X_i,Z_i,b^*,c^*)B(Z_i)^\top  L_b(Y_i|X_i,Z_i) X_i \mid Y_i>0]\\
&&\quad + \frac{1}{Np_N}\sum_{j=1,j\not=i}^N E[P(Y_i\mid X,Z)P(Y_j\mid X,Z) \ell_{\gamma\gamma}(Y_i|X_i,Z_i,b^*,c^*)\\
&&\quad\quad\quad\times B(Z_i)^\top  L_b(Y_j|X_j,Z_j) X_i\mid Y_i>0, Y_j>0]\\
&&\leq O(p_N K^{-m} + p_N^{1-\rho}).
\end{eqnarray*}
Similarly, we have 
\[
\frac{1}{Np_N}\sum_{i,j=1}^N E[\ell_{\gamma\sigma}(Y_i|X_i,Z_i,b_0,c_0)B(Z_i)^\top  L_c(Y_j|X_j,Z_j) X_i I(Y_i>0)]
\leq O(p_N K^{-m} + p_N^{1-\rho}).
\]
Consequently, we can prove $E[R_N(\hat{b},\hat{c})]=O(p_N K^{-m} + p_N^{1-\rho})=o(p_N^{-\rho})$, which completes the proof.
\end{proof}



\end{appendices}

\subsection*{Declarations}

\noindent{\bf Ethical Approval} There are no human and animal subjects in this manuscript. 

\vspace{3mm}

\noindent{\bf Conflict of Interest} The author declares that he has no confict of interest.

\vspace{3mm}

\noindent{\bf Funding} This research was financially supported by JSPS KAKENHI (Grant Nos. 22K11935 and 23K28043).

\vspace{3mm}

\noindent{\bf Acknowledgments}
We would like to thank the Editor, the Associate Editor, and the anonymous reviewers for their helpful comments and suggestions, which improved this report of our work.  
We also thank FASTEK JAPAN (www.fastekjapan.com) for English language editing.

\def\bibname{Reference}



\begin{thebibliography}{99}


\bibitem{beirlant04a}Beirlant, J. and Goegebeur, Y. (2004). Local polynomial maximum likelihood estimation for Pareto-type distributions. {\it Journal of Multivariate Analysis}. {\bf 89} 97--118.

\bibitem{beirlant04}Beirlant, J., Goegebeur, Y., Segers, J., and Teugels, J. (2004). {\it Statistics of extremes: Theory and applications}. John Wiley \& Sons.



\bibitem{chavez15}Chavez-Demoulin, V., Embrechts, P. and Hofert, M. (2015). An extreme value approach for modeling operational risk losses depending on covariates. {\it Journal of Risk Insurance}. {\bf 83} 735--776.

\bibitem{chavez05}Chavez-Demoulin, V. and Davison, A. C. (2005). Generalized additive modeling of sample extremes. {\it Journal of  the Royal Statistical Society Series C}. {\bf 54} 207-222.

\bibitem{claeskens09} Claeskens, G., Krivobokova, T. and Opsomer, J. D. (2009). Asymptotic properties of penalized spline estimators. {\it Biometrika}. {\bf 96} 529--544.

\bibitem{coles01} Coles, S. (2001). {\it An Introduction to statistical modeling of extreme values}. Springer, London.


\bibitem{daouia13}Daouia, A., Gardes, L., and Girard, S. (2013). On kernel smoothing for extremal quantile regression. {\it Bernoulli} {\bf 19} 2557--2589.

\bibitem{davison90} A. C. Davison, A.C. and Smith, R.L. (1990). Models for Exceedances Over High Thresholds. {\it Journal of the Royal Statistical Society: Series B}. {\bf 52},  393--425.

\bibitem{deboor01} de Boor, C. (2001). {\it A practical guide to splines}. Springer. New York. 

\bibitem{dehaan06}de Haan, L. and Ferreira, A. (2006). {\it Extreme value theory: An introduction}. New York: Springer-Verlag.

\bibitem{Drees01}Drees, H. (2001). Minimax risk bounds in extreme value theory. {\it Annals of Statistics}. {\bf 29}, 266--294.

\bibitem{Drees04}Drees, H., Ferreira, A. and de Haan, L. On maximum likelihood estimation of the extreme value index. {\it Annals of Applied Probability}. {\bf 14} 1179--1201. 



\bibitem{eilers1996} Eilers, P. H. C., and Marx, B. D. (1996). Flexible smoothing with B-splines and penalties. {\it Statistical Science}. {\bf 11} 89--121.

\bibitem{eilers2015} Eilers, P. H. C. (2015). Twenty years of P-splines. {\it SORT  -- Statistics and Operations Research Transactions}. {\bf 39} 149--186.

\bibitem{gnecco24} Gnecco, N., Terefe, E. M. and Engelke, S. (2024). Extremal random forests. {\it Journal of American Statistical Association.} 1--14. https://doi.org/10.1080/01621459.2023.2300522

\bibitem{green94}Green, P. J. and Silverman, B. W. (1994). {\it Nonparametric Regression and Generalized Linear Models: A roughness penalty approach}. Chapman \& Hall/CRC. New York.

\bibitem{hastie86}Hastie, T. J., and Tibshirani, R. J. (1986). Generalized additive models. {\it Statistical Science}. {\rm 1}, 297--310.

\bibitem{hastie90}Hastie, T. and Tibshirani, R. (1990). {\it Generalized additive models}. Chapman \& Hall/CRC. London.

\bibitem{hijort11}Hijort, N. L. and Pollard, D. (2011). Asymptotics for minimisers of convex processes. {\it arXiv::1107.3806}
 

\bibitem{li22}Li, R., Leng, C. and You, J. (2022). Semiparametric tail index regression. {\it Journal of Business \& Economic Statistics}. {\bf 40} 82--95. 

\bibitem{liu11}Liu, X., Wang, L. and Liang, H. (2011). Estimation and variable selection for semiparametric additive partial linear models. {\it Statistica Sinica}. {\bf 21} 1225--1248.


\bibitem{liu13}Liu, R., Yang, L. and H${\rm \ddot{a}}$rdle, W. K. (2013). Oracally efficient two-step estimation of generalized additive model. {\it Journal of American Statistical Association}. {\bf 108} 619--631.

\bibitem{liu17}Liu, R., H${\rm \ddot{a}}$rdle, W. K., and Zhang, G. (2017). Statistical inference for generalized additive partially linear models. {\it Jornal of Multivariate Analysis}. {\bf 162} 1--15.

\bibitem{marx98}Marx, B. D. and Eilers, P. H. C. (1998). Direct generalized additive modeling with penalized likelihood. {\it Computational Statistics \& Data Analysis}. {\bf 28} 193--209.

\bibitem{mhalla19}Mhalla, L., de Carvalho, M. and Chavez-Demoulin, V. (2019). Regression-type models for extreme dependence. {\it Scandinavian Journal of Statistics}. {\bf 46} 1141--1167.

\bibitem{mhalla19b}Mhalla, L., Opitz, T. and Chavez-Demoulin, V. (2019). Exceedance-based nonlinear regression of tail dependence. {\it Extremes}. {\bf 22} 523--552.

\bibitem{osullivan1986} O'Sullivan, F. (1986). A statistical perspective on ill-posed inverse problems (with discussion). \textit{Statistical Science}, {\bf 1} 502--518.


\bibitem{reiss07}Reiss, R. D. and Thomas, M. (2007). {\it Statistical analysis of extreme values: with applications to insurance, finance, hydrology and other fields}. Birkhaeuser

\bibitem{richards24} Richards, J. and Huser, R. (2024). Extreme quantile regression with deep learning. {\it arXiv:2404.09154}. 

\bibitem{ruppert09}Ruppert, D., Wand, M. P. and Carroll, R. J. (2009). Semiparametric regression during 2003-2007. {\it Electronic
Journal of Statistics} {\bf 3} 1193--1256.

\bibitem{smith85} Smith, R. L. (1985). Maximum likelihood estimation in a class of nonregular cases. {\it Biometrika}. {\bf 72}. 67--90.

\bibitem{smith87} Smith, R. L.(1987). Estimating tails of probability distributions. {\it The Annals of Statistics}. {\bf 15} 1174--1207.

\bibitem{van00} van de Geer, S. (2000). {\it Empirical Processes in $M$-Estimation}. Cambridge University Press, Cambridge.

\bibitem{vatter15}Vatter, T. and Chavez-Demoulin, V. (2015). Generalized additive models for conditional dependence structures. {\it Journal of Multivariate Analysis}. {\bf 141} 147--167.


\bibitem{wand2008}
Wand, M. P. and Ormerod, J. T. (2008). On semiparametric regression with O'Sullivan penalized splines. {\it Australian \& New Zealand Journal of Statistics}. {\bf  50}, 179--198.

\bibitem{wang09} Wang, H. and Tsai, C. L. (2009). Tail Index Regression. {\it Journal of the American Statistical Association}. {\bf 104} 1233--1240. 

\bibitem{wang07} Wang, L. and Yang, L. (2007). Spline-backfitted kernel smoothing of nonlinear additive autoregression model. {\it Annals of Statistics.} {\bf 35} 2474--2503. 

\bibitem{Wood2011} Wood, S. N. (2011). Fast stable restricted maximum likelihood and marginal likelihood estimation of semiparametric generalized linear models. {\it Journal of the Royal Statistical Society: Series B}, {\bf 73} 3--36.


\bibitem{wood17}Wood, S. N. (2017). {\it Generalized additive models: An Introduction with R, Second Edition}. Chapman \& Hall/CRC. London. 

\bibitem{xiao19}Xiao, L. (2019). Asymptotic theory of penalized spline. {\it Electronic Journal of Statistics}. {\bf 13} 747--794. 


\bibitem{yee15}Yee, T. W. (2015). {\it Vector generalized linear and additive models: With an implementation in R}. Springer-Verlag. New-York.

\bibitem{yee07}Yee, T. W. and Stephenson, A. G. (2007). Vector generalized linear and additive extreme value models. {\it Extremes}. {\bf 10} 1--9.

\bibitem{yoshida14} Yoshida, T. and Naito, K. (2014). Asymptotics for penalised splines in generalised additive models {\it Journal of Nonparametric Statistics}, \textbf{26}, 269--289.


\bibitem{youngman19} Youngman, B. D. (2019). Generalized additive models for exceedances of high thresholds with an application to return level estimation for U.S. wind gusts. {\it Journal of the American Statistical Association}. {\bf 114} 1865--1879. 

\bibitem{youngman22}Youngman, B. D. (2022). {evgam}: An {\sf R} package for generalized additive extreme value models. {\it Journal of Statistical Software}. {\bf 103} 1--26. 

\bibitem{zhou09}Zhou, C. (2009). Existence and consistency of the maximum likelihood estimator for the extreme value index. {\it Journal of Multivariate Analysis}. {\bf 100} 794--815.

\bibitem{zhou98}Zhou, S., Shen, X., and Wolfe, D.A. (1998). Local asymptotics for regression splines and confidence regions. {\it Annals of Statistics} {\bf 26} 1760--1782.
\end{thebibliography}

\end{document}